\documentclass[12pt]{article}
\usepackage[centertags]{amsmath}
\usepackage{amsfonts}
\usepackage{amssymb}
\usepackage{amsthm}
\usepackage{amsmath,mathrsfs}
\usepackage{amsmath,amscd}
\title{\textbf{The Generating Functional of the Kontsevich Integral and its Derivation as a Holonomy}}
\author{Renaud Gauthier \footnote{rg.mathematics@gmail.com} \\ \\Lycee Albert Camus}
\theoremstyle{definition}
\newtheorem*{acknowledgments}{Acknowledgments}
\newtheorem{ceedee}{Definition}[section]
\newtheorem{conc}[ceedee]{Definition}
\newtheorem{ex1}[ceedee]{Example}
\newtheorem{prodV}[ceedee]{Lemma}
\newtheorem{desc}{Lemma}[section]
\newtheorem{holo}{Proposition}[section]
\newtheorem{labeling}{Example}[section]
\newtheorem{topG}{Lemma}[section]
\newtheorem{TP}{Lemma}[section]
\newtheorem{tildeomega}{Lemma}[section]
\newtheorem{tilomlconn}[tildeomega]{Lemma}
\newtheorem{pullbacktiloml}{Lemma}[section]
\newtheorem{poundeq}{Lemma}[section]
\newtheorem{poundeqsimple}[poundeq]{Corollary}
\newtheorem{KZeq}{Lemma}[section]
\newcommand{\tilom}{\tilde{\omega}}
\newcommand{\tiloml}{\tilde{\omega}_{loc}}
\newcommand{\beq}{\begin{equation}}
\newcommand{\eeq}{\end{equation}}
\newcommand{\ess}{\;\setlength{\unitlength}{1.3mm}
\begin{picture}(5,4)
\put(1,1.5){\oval(2,1)[l]}
\put(3,0.5){\oval(2,1)[r]}
\put(0,0){\line(1,0){3}}
\put(1,1){\line(1,0){2}}
\put(1,2){\line(1,0){3}}
\end{picture}}
\newcommand{\dlogp}{\text{dlog}(\vartriangle \!\! z[P])}
\newcommand{\logp}{\log \vartriangle \!\! z[P]}
\newcommand{\barACNC}{\overline{\mathcal{A}}(C_N \mathbb{C})}
\newcommand{\barAXN}{\overline{\mathcal{A}}(X_N)}
\newcommand{\CNC}{C_N \mathbb{C}}
\newcommand{\barA}{\overline{\mathcal{A}}}
\newcommand{\dlog}{\text{dlog}}

\begin{document}
\maketitle
\begin{abstract}
We introduce an algebra bundle of chord diagrams over the configuration space of N points in the complex plane on which we put the Knizhnik-Zamolodchikov connection. For that particular connection, the holonomy along a loop in the base is shown to be generating the Kontsevich integral for that loop's associated braid.
\end{abstract}
\newpage

\section{The Kontsevich Integral}

We choose to define a knot to be a smooth embedding of $S^1$ into $S^3$ (or $\mathbb{R}^3$). One can define an equivalence relation on the category of knots called ambient isotopy, which we will denote by $\sim$. An isotopy between embeddings $\gamma_1,\gamma_2:X \rightarrow Y$ is a homotopy $\gamma:X\times I \rightarrow Y$ through embeddings ~\cite{GS}. An ambient isotopy between two knots $\gamma_0:S^1 \rightarrow S^3$ and $\gamma_1:S^1 \rightarrow S^3$ is an isotopy $\gamma: S^1 \times I \rightarrow S^3$ through diffeomorphisms $\Gamma:S^3\times I \rightarrow S^3$ such that $\Gamma_0=id_{S^3}$ and $\gamma_t=\Gamma_t \circ \gamma_0$ for all t. A knot is a link with only one component. We choose to define a link with $q$ components to be a smooth embedding of $\coprod_q S^1$ into $S^3$ (or $\mathbb{R}^3$). The above definition for isotopy carries over to the case of links.\\

In a first time, we will be interested in isotopy classes of knots since original definitions for the concepts that we will introduce presently where made in the case of knots. To distinguish one class from another, one needs a function on knots that takes different values on different classes, but which must of course be constant on equivalence classes, and such an object is rightfully called an isotopy invariant or invariant for short. Typically an invariant is valued in some abelian group $G$, and if we denote by $\mathbb{Z}Knots$ the abelian group generated by oriented knots, we write $\Gamma: \mathbb{Z}Knots /\!\sim \;\rightarrow G$ for an invariant $\Gamma$, it is a $G$-valued functional on the set of equivalence classes of knots. We can generalize this definition to the case of links: if $\mathbb{Z}Links /\!\sim$ is the set of equivalence classes of oriented links, $G$ is an abelian group, then a link invariant $\Gamma$ will be a map $\Gamma : \mathbb{Z}Links /\!\sim \;\rightarrow G$.\\


Of those invariants, Vassiliev invariants ~\cite{V} are of particular importance. Such invariants are functionals on $\mathbb{Z}Knots /\! \sim$ that are extended to be defined on singular knots whose singularities are transversal self-intersections. One first defines a positive crossing in the image of a knot to be:\\
\setlength{\unitlength}{0.5cm}
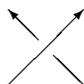
\begin{picture}(3,4)(-12,0)
\put(0,0){\vector(1,1){2}}
\put(2,0){\line(-1,1){0.8}}
\put(0.8,1.2){\vector(-1,1){0.8}}
\end{picture}\\ \\
where the arrows indicate the orientation on the portion of the knot that is being displayed. A negative crossing is represented as:\\
\setlength{\unitlength}{0.5cm}
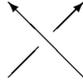
\begin{picture}(3,4)(-12,0)
\put(2,0){\vector(-1,1){2}}
\put(0,0){\line(1,1){0.8}}
\put(1.2,1.2){\vector(1,1){0.8}}
\end{picture}\\ \\
One represents a transversal self-intersection as follows:\\
\setlength{\unitlength}{0.5cm}
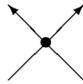
\begin{picture}(3,3)(-12,0)
\put(0,0){\vector(1,1){2}}
\put(2,0){\vector(-1,1){2}}
\put(1,1){\circle*{0.3}}
\end{picture}\\ \\
At this point one can extend invariants of knots to also be defined on singular knots, with the use of the relation:
\beq
\setlength{\unitlength}{0.5cm}
\begin{picture}(3,3)(4,0)
\put(-1,0.8){$V$}
\put(0,0){\vector(1,1){2}}
\put(2,0){\vector(-1,1){2}}
\put(1,1){\circle*{0.3}}
\put(2.5,0.8){$=$}
\put(4,0.8){$V$}
\put(5,0){\vector(1,1){2}}
\put(7,0){\line(-1,1){0.8}}
\put(5.8,1.2){\vector(-1,1){0.8}}
\put(8,0.8){$-$}
\put(9.5,0.8){$V$}
\put(12.5,0){\vector(-1,1){2}}
\put(10.5,0){\line(1,1){0.8}}
\put(11.7,1.2){\vector(1,1){0.8}}
\end{picture} \label{Vass}
\eeq
where $V$ is any link invariant. By iterating this procedure, one can extend a link invariant to be defined on knots with multiple double-crossings. If a link invariant $V$ has its extension vanishing on knots with more than $m$ self-intersections, one says that $V$ is a Vassiliev invariant of type $m$.\\
\setlength{\unitlength}{0.5cm}
\begin{picture}(3,4)(-8,0)
\put(-1,0.8){$V$}
\put(0,0.9){$($}
\put(1,0){\vector(1,1){2}}
\put(3,0){\vector(-1,1){2}}
\put(2,1){\circle*{0.3}}
\put(4,0){\vector(1,1){2}}
\put(6,0){\vector(-1,1){2}}
\put(5,1){\circle*{0.3}}
\put(7,1){\circle*{0.2}}
\put(7.5,1){\circle*{0.2}}
\put(8,1){\circle*{0.2}}
\put(9,0){\vector(1,1){2}}
\put(11,0){\vector(-1,1){2}}
\put(10,1){\circle*{0.3}}
\put(12,0.8){$)$}
\put(12.5,0.8){$=0$}
\end{picture}\\ \\
where the argument of $V$ above has more than $m$ self-intersections. Due to the local nature of the extension of knot invariants to singular knots, we see that we can generalize the definition of a Vassiliev invariant to the case of links. A Vassiliev invariant of links will be said to be of type $m$ if it evaluates to zero on any singular link with more than $m$ double-crossings (not necessarily same component intersections).\\


Among the Vassiliev invariants of finite type $m$, one has the degree $m$ part of the Kontsevich integral first introduced in ~\cite{K}. One should note that the Kontsevich integral as it was initially defined is not strictly speaking a knot invariant, but once it is corrected as we shall do below, then it becomes an honest knot invariant, and its final corrected form is a universal Vassiliev invariant in the sense that every finite type Vassiliev invariant factors through it ~\cite{K} ~\cite{DBN} ~\cite{DBN2} ~\cite{ChDu}.\\


Before introducing this integral, we define the algebra $\mathcal{A}$ ~\cite{K} in which the integral takes its values. For a singular oriented knot whose only singularities are transversal self-intersections, the preimage of each singular crossing under the embedding map defining the knot yields a pair of distinct points on $S^1$. Each singular point in the image therefore yields a pair of points on $S^1$ that are conventionally connected by a chord for book keeping purposes. A knot with $m$ singular points will yield $m$ distinct chords on $S^1$. One refers to such a circle with $m$ chords on it as a chord diagram of degree $m$, the degree being the number of chords. The support of the graph is an oriented $S^1$, and it is regarded up to orientation preserving diffeomorphisms of the circle. More generally, for a singular oriented link all of whose singularities are double-crossings, preimages of each singular crossing under the embedding map defining the link yield pairs of distinct points on possibly different circles depending on whether the double crossing was on a same component or between different components of the link. We also connect points making a pair by a chord. A link with $m$ singular points will yield $m$ chords on $\coprod S^1$. We call such a graph a chord diagram. The support is $\coprod S^1$ regarded up to orientation preserving diffeomorphism of each $S^1$.\\

One denotes by $\mathcal{D}$ the complex vector space spanned by chord diagrams with support $S^1$. There is a grading on $\mathcal{D}$ given by the number of chords featured in a diagram. If $\mathcal{D}^{(m)}$ denotes the subspace of chord diagrams of degree $m$, then one writes:
\beq
\mathcal{D}=\oplus_{m\geq 0} \mathcal{D}^{(m)}
\eeq
One quotients this space by the 4-T relation which locally looks like:\\
\setlength{\unitlength}{1cm}
\begin{picture}(1,2)(-1,-0.5)
\multiput(0,0.75)(0.2,0){5}{\line(1,0){0.1}}
\multiput(1,0.25)(0.2,0){5}{\line(1,0){0.1}}
\put(0,0.9){\vector(0,1){0.2}}
\put(1,0.9){\vector(0,1){0.2}}
\put(2,0.9){\vector(0,1){0.2}}
\linethickness{0.3mm}
\put(0,0){\line(0,1){1}}
\put(1,0){\line(0,1){1}}
\put(2,0){\line(0,1){1}}
\put(2.5,0.5){$+$}
\end{picture}
\setlength{\unitlength}{1cm}
\begin{picture}(1,2)(-3,-0.5)
\multiput(0,0.75)(0.2,0){5}{\line(1,0){0.1}}
\multiput(0,0.25)(0.2,0){10}{\line(1,0){0.1}}
\put(0,0.9){\vector(0,1){0.2}}
\put(1,0.9){\vector(0,1){0.2}}
\put(2,0.9){\vector(0,1){0.2}}
\linethickness{0.3mm}
\put(0,0){\line(0,1){1}}
\put(1,0){\line(0,1){1}}
\put(2,0){\line(0,1){1}}
\put(2.5,0.5){$=$}
\end{picture}
\setlength{\unitlength}{1cm}
\begin{picture}(1,2)(-5,-0.5)
\multiput(0,0.25)(0.2,0){5}{\line(1,0){0.1}}
\multiput(1,0.75)(0.2,0){5}{\line(1,0){0.1}}
\put(0,0.9){\vector(0,1){0.2}}
\put(1,0.9){\vector(0,1){0.2}}
\put(2,0.9){\vector(0,1){0.2}}
\linethickness{0.3mm}
\put(0,0){\line(0,1){1}}
\put(1,0){\line(0,1){1}}
\put(2,0){\line(0,1){1}}
\put(2.5,0.5){$+$}
\end{picture}
\setlength{\unitlength}{1cm}
\begin{picture}(1,2)(-7,-0.5)
\multiput(0,0.25)(0.2,0){5}{\line(1,0){0.1}}
\multiput(0,0.75)(0.2,0){10}{\line(1,0){0.1}}
\put(0,0.9){\vector(0,1){0.2}}
\put(1,0.9){\vector(0,1){0.2}}
\put(2,0.9){\vector(0,1){0.2}}
\linethickness{0.3mm}
\put(0,0){\line(0,1){1}}
\put(1,0){\line(0,1){1}}
\put(2,0){\line(0,1){1}}
\end{picture}\\ \\
where solid lines are intervals on $S^1$ on which a chord foot rests, and arrows indicate the orientation of each strand. One further quotients this space by the framing independence relation: if a chord diagram has a chord forming an arc on $S^1$ with no other chord ending in between its feet, then the chord diagram is set to zero. The resulting quotient space is the complex vector space generated by chord diagrams mod the 4-T relation and framing independence and is denoted by $\mathcal{A}$. The grading of $\mathcal{D}$ is preserved by the quotient, inducing a grading on $\mathcal{A}$:
\beq
\mathcal{A}=\oplus_{m \geq 0}\mathcal{A}^{(m)}
\eeq
where $\mathcal{A}^{(m)}$ is obtained from $\mathcal{D}^{(m)}$ upon modding out by the 4-T and the framing independence relations. All this carries over to the case of links by formally extending the 4-T relation to the case of $q$ disjoint copies of the circle in the case of a $q$-components link, and the resulting $\mathbb{C}$-vector space will be denoted $\mathcal{A}(\coprod_q S^1)$.\\

The connected sum of circles can be extended to chorded circles, thereby defining a product on $\mathcal{A}$, making it into an associative and commutative algebra. The Kontsevich integral will be valued in the graded completion $\overline{\mathcal{A}}=\prod_{m \geq 0}\mathcal{A}^{(m)}$ of the algebra $\mathcal{A}$.\\

As far as knots are concerned, we will work with Morse knots and geometric tangles, and for that purpose one considers the following decomposition of $\mathbb{R}^3$ as the product of the complex plane and the real line: $\mathbb{R}^3=\mathbb{R}^2\times \mathbb{R} \simeq \mathbb{C}\times \mathbb{R}$, with local coordinates $z$ on the complex plane and $t$ on the real line for time. A morse knot $K$ is such that $t\circ K$ is a Morse function on $S^1$. If one denotes by $Z$ the Kontsevich integral functional on knots, if $K$ is a Morse knot, one defines ~\cite{K}, ~\cite{DBN}:
\beq
Z(K):=\sum_{m\geq 0} \frac{1}{(2 \pi i)^m}\int_{t_{min}< t_1<...<t_m<t_{max}}\sum_{P\; applicable}(-1)^{\varepsilon(P)}D_P\prod_{1\leq i \leq m}\dlog \vartriangle \!\!z[P_i] \label{IK}
\eeq
where $t_{min}$ and $t_{max}$ are the min and max values of $t$ on $K$ respectively, $P$ is an $m$-vector each entry of which is a pair of points on the image of the knot $K$, $P=(P_1,...,P_m)$, where the $i$-th entry $P_i$ corresponds to a pair of points on the knot. One refers to such $P$'s as pairings. If we further situate these paired points at some height $t_i$, and denote these two points by $z_i$ and $z'_i$, then we define $\vartriangle \!\! z[P_i]:=z_i-z'_i$. We denote by $K_P$ the knot $K$ with $m$ pairs of points placed on it following the prescription given by $P$, along with chords connecting such points at a same height. A pairing is said to be applicable if each entry is a pair of two distinct points on the knot, at the same height ~\cite{DBN}. We will assume that all chords are horizontal on knots and will drop the adjective applicable, simply referring to $P$'s as pairings. One denotes by $\varepsilon(P)$ the number of those points ending on portions of $K$ that are locally oriented down. For example if $P=(z(t),z'(t))$ and $K$ is decreasing at $z(t)$, then it will contribute 1 to $\varepsilon(P)$. One also define the length of $P$ to be $|P|$, the number of pairings it is a combination of. If we denote by $\iota_K$ the embedding defining the knot $K$ then $D_P$ is defined to be the chord diagram one obtains by taking the inverse image of $K_P$ under $\iota_K$, that is $D_P=\iota_K^{-1} K_P$. This generalizes immediately to the case of Morse links, and in this case the geometric coefficient will not be an element of $\barA$ but will be an element of $\barA (\coprod_q S^1)$ if the argument of $Z$ is a $q$-components link.\\

Now if one wants to make this integral into a true knot invariant, then one corrects it as follows. Consider the embedding in $S^3$ of the trivial knot as:\\
\setlength{\unitlength}{0.5cm}
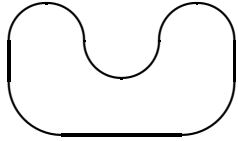
\begin{picture}(5,6)(-11,0)
\thicklines
\put(1,3){\oval(2,2)[t]}
\put(3,3){\oval(2,2)[b]}
\put(5,3){\oval(2,2)[t]}
\put(3,3){\oval(6,5)[b]}
\end{picture}\\ \\
Consider the following correction ~\cite{K}:
\beq
\hat{Z}:=Z(\:\setlength{\unitlength}{0.1cm}
\begin{picture}(5,6)
\thicklines
\put(1,2){\oval(2,2)[t]}
\put(3,2){\oval(2,2)[b]}
\put(5,2){\oval(2,2)[t]}
\put(3,2){\oval(6,5)[b]}
\end{picture}\;)^{-m}.Z \label{modK}
\eeq
where the dot is the product on chord diagrams extended by linearity, and $m$ is a function that captures the number of maximal points of any knot $K$ that is used as an argument of $Z$. In the case of links, there will be one such correction for each component of the link, with a power $m_i$ on the correction term for the $i$-th component, where $m_i$ is the number of maximal points of the $i$-th link component. Equivalently, $\hat{Z}(L)$ is the same as $Z(L)$ save that every $i$-th link component in the expression for $Z(L)$ is multiplied by $\nu^{m_i}$, $1 \leq i \leq q$. In what follows we will be working with $Z$ only and not its corrected form $\hat{Z}$.\\

The Kontsevich integral $Z(K)$ of a knot $K$ can be computed as in ~\eqref{IK}, which can easily be generalized to the case of links. However, it can also be computed by first putting the knot $K$ in braid position ~\cite{A}. This is done as follows. We can generalize the Kontsevich integral to the case of tangles ~\cite{DBN}, and in particular to that of braids. By closing the braid to recover the knot, the Kontsevich integral of the braid then yields $Z(K)$. In this work we will mainly focus on the integral of braids. Braids close to links. This enables one to define the integral $Z(L)$ of links $L$ whose coefficients are valued in chord diagrams with support on the link $L$. By taking the preimage of such chord diagrams under the embedding map defining the link $L$, the coefficients become valued in $\barA (\coprod S^1)$ where we have as many copies of $S^1$ as we have components in the link. This is how we are naturally led to consider the Kontsevich integral of links, and we will follow the above procedure for computing such an integral.\\

In section 2, we introduce the configuration space of $N$ unordered points in the complex plane, the natural setting for studying braids. In section 3 we present the general notion of chord diagrams. After introduction of this fundamental material we present the background work done on the Kontsevich integral and how it was argued that it was the expansion of some holonomy in section 4. Section 5 presents a detailed proof that the knot functional generating the Kontsevich integral is exactly the holonomy for some connection on a bundle of chord diagrams over the configuration space of N points. Section 6 is a careful derivation of the Knizhnik-Zamolodchikov equation from first principles. \\

\begin{acknowledgments}
The author would like to thank D.Yetter for valuable discussions regarding certain aspects of this work.
\end{acknowledgments}

\section{The configuration space of $N$ points in the plane}
A link in $S^3$ is ambient isotopic to a closed braid ~\cite{A} ~\cite{JB}, so that one can deform a link into a braid part, outside of which all its strands are parallel. For a given link, let $N$ be the number of strands of its braid part. $N$ will depend on the link we have chosen. The transversal intersection of these $N$ strands with the complex plane will yield a set of $N$ distinct points, each point resulting from the intersection of one strand with this plane. It is natural then to study, for any given $N$, the space $X_N$ defined as the configuration space of $N$ distinct unordered points in the complex plane:
\beq
X_N:=\{(z_1,...,z_n) \in \mathbb{C}^N | z_i=z_j \Rightarrow i=j\}/ S_N=(\mathbb{C}^N-\Delta)/S_N
\eeq
where $S_N$ is the permutation group on $N$ elements and $\Delta$ is the big diagonal in $\mathbb{C}^N$. The labeling of points of $X_N$ is not induced by any ordering on $\mathbb{C}^N$ but rather is a way to locate the $N$ points in the complex plane whose collection defines a single point of $X_N$. We will sometimes write $\sum_{1 \leq i \leq N}[z_i]$ instead of $\{z_1,...,z_N\}$ to represent points in configuration space. The points $z_1,...,z_N$ of the complex plane defining a point $Z=\sum_{1 \leq i \leq N}[z_i]$ of $X_N$ will be referred to as the $N$ defining points of $Z$. We consider the topology $\tau$ on $X_N$ generated by open sets of the form $U=\{U_1,...,U_N\}$ where the $U_i,\,1 \leq i \leq N$ are non-overlapping open sets in the complex plane. We will also refer to those open sets $U_1,...,U_N$ as the $N$ defining open sets of the open set $U$ of $X_N$. \\

We review the basic terminology pertaining to braids as presented in ~\cite{JB} since we will work extensively with braids in what follows. The pure braid group of $\mathbb{C}^N$ is defined to be $\pi_1 (\mathbb{C}^N - \Delta)$, and the braid group of $\mathbb{C}^N$ is defined to be $\pi_1 (X_N)$. A braid is an element of this latter group. If $q$ denotes the regular projection map from $\mathbb{C}^N - \Delta$ to $X_N$, $Z=(z_1,...,z_N) \in \mathbb{C}^N - \Delta$, $qZ \in X_N$, then $\gamma \in \pi_1 (X_N ,qZ)$ based at $qZ$ is given by a loop $\gamma =\{\gamma_1,...,\gamma_N \}$ which lifts uniquely to a path in $\mathbb{C}^N - \Delta$ based at $Z$ that without loss of generality we will denote by the same letter $\gamma$. Then we have $\gamma =(\gamma_1 ,...,\gamma_N )$. The graph of the $i$-th coordinate of $\gamma$ is defined to be $\Gamma_i := \{(\gamma_i (t),t)\;|\;t \in I \}$, $1 \leq i \leq N$. Each such graph $\Gamma_i$ defines an arc $\tilde{\gamma}_i \in \mathbb{C} \times I$ and $\tilde{\gamma}:=\cup_{1 \leq i \leq N} \tilde{\gamma}_i \in \mathbb{C}\times I $ is called a geometric braid, which we will refer to as the lift of $\gamma$. As such it is open, and its closure is a closed braid.

\section{Chord diagrams}
We will be interested in considering chord diagrams with support on a point, and later on chord diagrams with support on a geometric braid $\tilde{\gamma}$ in $\mathbb{C} \times I$, so for that purpose one considers a more general definition of chord diagrams than the one presented in the introduction which was sufficient to discuss the Kontsevich integral of knots.
\begin{ceedee}[\cite{lemu}]
Let $X$ be a one dimensional, compact, oriented, smooth manifold with numbered components. A chord diagram with support on $X$ is a set of finitely many unordered pairs of distinct non-boundary points on $X$ defined modulo orientation and component preserving homeomorphisms. One realizes each pair geometrically by drawing a dashed line, or chord, stretching from one point to the other. One denotes by $\mathcal{A}(X)$ the $\mathbb{C}$-vector space spanned by chord diagrams with support on $X$ modulo the framing indepence relation as well as the 4-T relation: if $i$, $j$ and $k$ are indices for components of $X$ on which chords are ending, then locally the 4-T relation can be written:

\setlength{\unitlength}{1cm}
\begin{picture}(1,2)(0,-0.5)
\multiput(0,0.75)(0.2,0){5}{\line(1,0){0.1}}
\multiput(1,0.25)(0.2,0){5}{\line(1,0){0.1}}
\put(0,0.9){\vector(0,1){0.2}}
\put(1,0.9){\vector(0,1){0.2}}
\put(2,0.9){\vector(0,1){0.2}}
\linethickness{0.3mm}
\put(0,0){\line(0,1){1}}
\put(1,0){\line(0,1){1}}
\put(2,0){\line(0,1){1}}
\put(0.1,-0.2){$i$}
\put(1.1,-0.2){$j$}
\put(2.1,-0.2){$k$}
\put(2.5,0.5){$+$}
\end{picture}
\setlength{\unitlength}{1cm}
\begin{picture}(1,2)(-2,-0.5)
\multiput(0,0.75)(0.2,0){5}{\line(1,0){0.1}}
\multiput(0,0.25)(0.2,0){10}{\line(1,0){0.1}}
\put(0,0.9){\vector(0,1){0.2}}
\put(1,0.9){\vector(0,1){0.2}}
\put(2,0.9){\vector(0,1){0.2}}
\linethickness{0.3mm}
\put(0,0){\line(0,1){1}}
\put(1,0){\line(0,1){1}}
\put(2,0){\line(0,1){1}}
\put(0.1,-0.2){$i$}
\put(1.1,-0.2){$j$}
\put(2.1,-0.2){$k$}
\put(2.5,0.5){$=$}
\end{picture}
\setlength{\unitlength}{1cm}
\begin{picture}(1,2)(-4,-0.5)
\multiput(0,0.25)(0.2,0){5}{\line(1,0){0.1}}
\multiput(1,0.75)(0.2,0){5}{\line(1,0){0.1}}
\put(0,0.9){\vector(0,1){0.2}}
\put(1,0.9){\vector(0,1){0.2}}
\put(2,0.9){\vector(0,1){0.2}}
\linethickness{0.3mm}
\put(0,0){\line(0,1){1}}
\put(1,0){\line(0,1){1}}
\put(2,0){\line(0,1){1}}
\put(0.1,-0.2){$i$}
\put(1.1,-0.2){$j$}
\put(2.1,-0.2){$k$}
\put(2.5,0.5){$+$}
\end{picture}
\setlength{\unitlength}{1cm}
\begin{picture}(1,2)(-6,-0.5)
\multiput(0,0.25)(0.2,0){5}{\line(1,0){0.1}}
\multiput(0,0.75)(0.2,0){10}{\line(1,0){0.1}}
\put(0,0.9){\vector(0,1){0.2}}
\put(1,0.9){\vector(0,1){0.2}}
\put(2,0.9){\vector(0,1){0.2}}
\linethickness{0.3mm}
\put(0,0){\line(0,1){1}}
\put(1,0){\line(0,1){1}}
\put(2,0){\line(0,1){1}}
\put(0.1,-0.2){$i$}
\put(1.1,-0.2){$j$}
\put(2.1,-0.2){$k$}
\end{picture}

One defines the degree of a chord diagram to be the number of chords a chord diagram has, and we call it the chord degree. This induces a graded decomposition of the space $\mathcal{A}(X)$:
\beq
\mathcal{A}(X)=\bigoplus_{m\geq 0}\mathcal{A}^{(m)}(X)
\eeq
where $\mathcal{A}^{(m)}(X)$ is the $\mathbb{C}$-vector space of chord diagrams of degree $m$ with support on $X$. We write $\overline{\mathcal{A}}(X)$ for the graded completion of $\mathcal{A}(X)$.
\end{ceedee}
One is especially interested in the case where $X$ is a geometric braid $\tilde{\gamma} \in \mathbb{C} \times I$ corresponding to some loop $\gamma$ in $X_N$. The strands are oriented up, $t=0$ being the bottom plane of the space $\mathbb{C}\times I$ in which the braid is embedded, $t=1$ corresponding to the top plane. Since indices for pairings match those for the times at which they are located, chords will be ordered from the bottom up. For $m=1$, a chord will stretch between two strands, say the strands indexed by $i$ and $j$, and we will denote such a chord diagram by  $|ij\rangle \in \mathcal{A}(\tilde{\gamma})$, corresponding to the  pairing $(ij)$ in this case. If we want to insist that the skeleton of the chord diagram is a given geometric braid $\tilde{\gamma}$ then we write $|ij\rangle (\tilde{\gamma})$. In certain situations it will be necessary to also indicate at which point along the braid is the chord situated for location purposes. Once we have $|ij\rangle (\tilde{\gamma})$, it is sufficient to have the height $t \in I$ at which we have to place the chord $|ij\rangle$ on $\tilde{\gamma}$ and $|ij\rangle (\tilde{\gamma})(t)$ is defined to be a chord between the $i$-th and $j$-th strands of $\tilde{\gamma}$ at height $t$, or equivalently a chord between $(\gamma_i (t),t)$ and $(\gamma_j (t),t)$. In that case we work with a representative of the class defining the chord diagram $|ij\rangle (\tilde{\gamma})$. We define a non-commutative, associative product $\ess$ on $\mathcal{A}(X)$ for some geometric braid $X$ as follows:

\begin{conc}
Let $X$ be a geometric braid with $N$ strands. We define the $\ess$-product of two chord diagrams $D_1 \in \mathcal{A}^{(m_1)}(X)$ and $D_2 \in \mathcal{A}^{(m_2)}(X)$, and we denote it by $D_1 \ess D_2$ to be the chord diagram in $\mathcal{A}^{(m_1+m_2)}(X)$ whose chords are, from the top, those of $D_1$ followed by those of $D_2$. Equivalently, $D_1 \ess D_2$ is the concatenation of the chord diagrams $D_1$ and $D_2$ with support the trivial braid, followed by the concatenation with $X$. The identity for this product is $X$ itself, the empty chord diagram. We will sometimes use a vertical notation for such products to make the reading easier. For example, we will sometimes write
\beq
\begin{array}{c}
  D_1 \\
  \ess \\
  D_2
\end{array}
\eeq
instead of $D_1 \ess D_2$. By definition, this product is associative, and non-commutative.
\end{conc}
\begin{ex1}
Pick $X$ as the following geometric braid on 3 strands:\\
\setlength{\unitlength}{0.5cm}
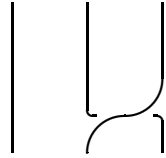
\begin{picture}(5,6)(-10,0)
\thicklines
\multiput(3,3)(2,0){3}{\line(0,1){2}}
\put(3,1){\line(0,1){2}}
\put(6,3){\oval(2,2)[br]}
\put(6.75,1){\oval(0.5,2)[tr]}
\put(5.25,3){\oval(0.5,2)[bl]}
\put(6,1){\oval(2,2)[tl]}
\end{picture}\\ \\
and:\\
\setlength{\unitlength}{0.5cm}
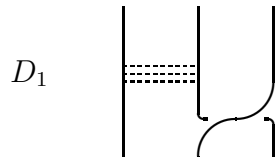
\begin{picture}(5,6)(-10,0)
\put(0,3){$D_1$}
\multiput(3,3)(0.2,0){10}{\line(1,0){0.1}}
\multiput(3,3.2)(0.2,0){10}{\line(1,0){0.1}}
\multiput(3,3.4)(0.2,0){10}{\line(1,0){0.1}}
\thicklines
\multiput(3,3)(2,0){3}{\line(0,1){2}}
\put(3,1){\line(0,1){2}}
\put(6,3){\oval(2,2)[br]}
\put(6.75,1){\oval(0.5,2)[tr]}
\put(5.25,3){\oval(0.5,2)[bl]}
\put(6,1){\oval(2,2)[tl]}
\end{picture}\\ \\
\setlength{\unitlength}{0.5cm}
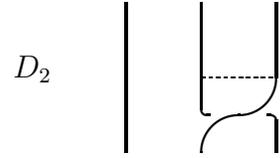
\begin{picture}(5,6)(-10,0)
\put(0,3){$D_2$}
\multiput(5,3)(0.2,0){10}{\line(1,0){0.1}}
\thicklines
\multiput(3,3)(2,0){3}{\line(0,1){2}}
\put(3,1){\line(0,1){2}}
\put(6,3){\oval(2,2)[br]}
\put(6.75,1){\oval(0.5,2)[tr]}
\put(5.25,3){\oval(0.5,2)[bl]}
\put(6,1){\oval(2,2)[tl]}
\end{picture}\\ \\
\setlength{\unitlength}{0.5cm}
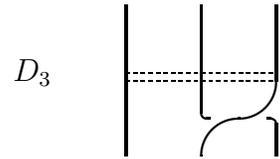
\begin{picture}(5,6)(-10,0)
\put(0,3){$D_3$}
\multiput(3,3)(0.2,0){20}{\line(1,0){0.1}}
\multiput(3,3.2)(0.2,0){20}{\line(1,0){0.1}}
\thicklines
\multiput(3,3)(2,0){3}{\line(0,1){2}}
\put(3,1){\line(0,1){2}}
\put(6,3){\oval(2,2)[br]}
\put(6.75,1){\oval(0.5,2)[tr]}
\put(5.25,3){\oval(0.5,2)[bl]}
\put(6,1){\oval(2,2)[tl]}
\end{picture}\\ \\
then we have:\\
\setlength{\unitlength}{0.5cm}
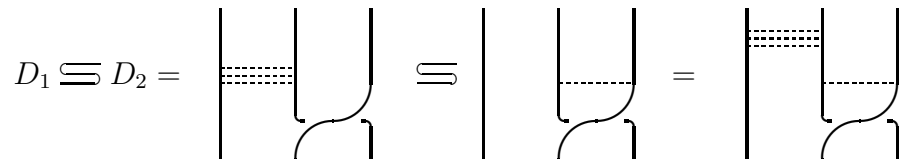
\begin{picture}(22,6)(-4,0)
\put(-2.5,3){$D_1 \ess D_2=$}
\multiput(3,3)(0.2,0){10}{\line(1,0){0.1}}
\multiput(3,3.2)(0.2,0){10}{\line(1,0){0.1}}
\multiput(3,3.4)(0.2,0){10}{\line(1,0){0.1}}
\put(8,3){$\ess$}
\multiput(12,3)(0.2,0){10}{\line(1,0){0.1}}
\put(15,3){$=$}
\multiput(17,4)(0.2,0){10}{\line(1,0){0.1}}
\multiput(17,4.2)(0.2,0){10}{\line(1,0){0.1}}
\multiput(17,4.4)(0.2,0){10}{\line(1,0){0.1}}
\multiput(19,3)(0.2,0){10}{\line(1,0){0.1}}
\thicklines
\multiput(3,3)(2,0){3}{\line(0,1){2}}
\put(3,1){\line(0,1){2}}
\put(6,3){\oval(2,2)[br]}
\put(6.75,1){\oval(0.5,2)[tr]}
\put(5.25,3){\oval(0.5,2)[bl]}
\put(6,1){\oval(2,2)[tl]}
\multiput(10,3)(2,0){3}{\line(0,1){2}}
\put(10,1){\line(0,1){2}}
\put(13,3){\oval(2,2)[br]}
\put(13.75,1){\oval(0.5,2)[tr]}
\put(12.25,3){\oval(0.5,2)[bl]}
\put(13,1){\oval(2,2)[tl]}
\multiput(17,3)(2,0){3}{\line(0,1){2}}
\put(17,1){\line(0,1){2}}
\put(20,3){\oval(2,2)[br]}
\put(20.75,1){\oval(0.5,2)[tr]}
\put(19.25,3){\oval(0.5,2)[bl]}
\put(20,1){\oval(2,2)[tl]}
\end{picture}\\ \\
whereas:\\
\setlength{\unitlength}{0.5cm}
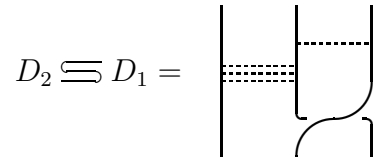
\begin{picture}(22,6)(-9,0)
\put(-2.5,3){$D_2 \ess D_1=$}
\multiput(3,3)(0.2,0){10}{\line(1,0){0.1}}
\multiput(3,3.2)(0.2,0){10}{\line(1,0){0.1}}
\multiput(3,3.4)(0.2,0){10}{\line(1,0){0.1}}
\multiput(5,4)(0.2,0){10}{\line(1,0){0.1}}
\thicklines
\multiput(3,3)(2,0){3}{\line(0,1){2}}
\put(3,1){\line(0,1){2}}
\put(6,3){\oval(2,2)[br]}
\put(6.75,1){\oval(0.5,2)[tr]}
\put(5.25,3){\oval(0.5,2)[bl]}
\put(6,1){\oval(2,2)[tl]}
\end{picture}\\ \\
thereby illustrating the non-commutative nature of the $\ess$-product.
For the associativity:\\
\setlength{\unitlength}{0.5cm}
\begin{picture}(22,6)(-6,0)
\put(-5,3){$(D_1 \ess D_2 ) \ess D_3 =$}
\multiput(3,4)(0.2,0){10}{\line(1,0){0.1}}
\multiput(3,4.2)(0.2,0){10}{\line(1,0){0.1}}
\multiput(3,4.4)(0.2,0){10}{\line(1,0){0.1}}
\multiput(5,3)(0.2,0){10}{\line(1,0){0.1}}
\put(8,3){$\ess$}
\multiput(10,3)(0.2,0){20}{\line(1,0){0.1}}
\multiput(10,3.2)(0.2,0){20}{\line(1,0){0.1}}
\put(15,3){$=$}
\multiput(17,4)(0.2,0){10}{\line(1,0){0.1}}
\multiput(17,4.2)(0.2,0){10}{\line(1,0){0.1}}
\multiput(17,4.4)(0.2,0){10}{\line(1,0){0.1}}
\multiput(19,3.5)(0.2,0){10}{\line(1,0){0.1}}
\multiput(17,3)(0.2,0){20}{\line(1,0){0.1}}
\multiput(17,3.2)(0.2,0){20}{\line(1,0){0.1}}
\thicklines
\multiput(3,3)(2,0){3}{\line(0,1){2}}
\put(3,1){\line(0,1){2}}
\put(6,3){\oval(2,2)[br]}
\put(6.75,1){\oval(0.5,2)[tr]}
\put(5.25,3){\oval(0.5,2)[bl]}
\put(6,1){\oval(2,2)[tl]}
\multiput(10,3)(2,0){3}{\line(0,1){2}}
\put(10,1){\line(0,1){2}}
\put(13,3){\oval(2,2)[br]}
\put(13.75,1){\oval(0.5,2)[tr]}
\put(12.25,3){\oval(0.5,2)[bl]}
\put(13,1){\oval(2,2)[tl]}
\multiput(17,3)(2,0){3}{\line(0,1){2}}
\put(17,1){\line(0,1){2}}
\put(20,3){\oval(2,2)[br]}
\put(20.75,1){\oval(0.5,2)[tr]}
\put(19.25,3){\oval(0.5,2)[bl]}
\put(20,1){\oval(2,2)[tl]}
\end{picture}\\ \\
but we also have:\\
\setlength{\unitlength}{0.5cm}
\begin{picture}(22,6)(-6,0)
\put(-5,3){$D_1 \ess(D_2\ess D_3)=$}
\multiput(3,3)(0.2,0){10}{\line(1,0){0.1}}
\multiput(3,3.2)(0.2,0){10}{\line(1,0){0.1}}
\multiput(3,3.4)(0.2,0){10}{\line(1,0){0.1}}
\multiput(12,4)(0.2,0){10}{\line(1,0){0.1}}
\put(8,3){$\ess$}
\multiput(10,3)(0.2,0){20}{\line(1,0){0.1}}
\multiput(10,3.2)(0.2,0){20}{\line(1,0){0.1}}
\put(15,3){$=$}
\multiput(17,4)(0.2,0){10}{\line(1,0){0.1}}
\multiput(17,4.2)(0.2,0){10}{\line(1,0){0.1}}
\multiput(17,4.4)(0.2,0){10}{\line(1,0){0.1}}
\multiput(19,3.5)(0.2,0){10}{\line(1,0){0.1}}
\multiput(17,3)(0.2,0){20}{\line(1,0){0.1}}
\multiput(17,3.2)(0.2,0){20}{\line(1,0){0.1}}
\thicklines
\multiput(3,3)(2,0){3}{\line(0,1){2}}
\put(3,1){\line(0,1){2}}
\put(6,3){\oval(2,2)[br]}
\put(6.75,1){\oval(0.5,2)[tr]}
\put(5.25,3){\oval(0.5,2)[bl]}
\put(6,1){\oval(2,2)[tl]}
\multiput(10,3)(2,0){3}{\line(0,1){2}}
\put(10,1){\line(0,1){2}}
\put(13,3){\oval(2,2)[br]}
\put(13.75,1){\oval(0.5,2)[tr]}
\put(12.25,3){\oval(0.5,2)[bl]}
\put(13,1){\oval(2,2)[tl]}
\multiput(17,3)(2,0){3}{\line(0,1){2}}
\put(17,1){\line(0,1){2}}
\put(20,3){\oval(2,2)[br]}
\put(20.75,1){\oval(0.5,2)[tr]}
\put(19.25,3){\oval(0.5,2)[bl]}
\put(20,1){\oval(2,2)[tl]}
\end{picture}\\ \\
In words, the product $\ess$ is just a superposition map. Extend $\ess$ by linearity, thereby making $(\mathcal{A}(X), +, \ess)$ into a $\mathbb{C}$-algebra.
\end{ex1}

For $X$ a geometric braid, we have the following fact for $(\mathcal{A}(X), +, \ess)$:
\beq
\mathcal{A}^{(m)}(X)=\begin{array}{c}
                                  \mathcal{A}^{(1)}(X) \\
                                  \ess \\
                                  \vdots \\
                                  \ess \\
                                  \mathcal{A}^{(1)}(X)
                                \end{array}\qquad \Bigg\}m
                                \label{stack}
\eeq
\\

We will be interested in chord diagrams supported at a point of $C_N\mathbb{C}:=\mathbb{C}^N - \Delta$, or a point of $X_N$. For a point $Z=\{z_1,...,z_N\} \in X_N$, some $P=(k,l)$, $1 \leq k \neq l \leq N$, $|P\rangle(Z) \in \mathcal{A}(Z)$ is a chord between $z_k$ and $z_l$ in $X_N$. We denote by $\mathcal{A}(X_N)$ the complex vector space spanned by all such elements, and by $\overline{\mathcal{A}}(X_N)$ its graded completion.\\

We will also be interested in working with elements of $\mathcal{A}^{(1)}(X) \otimes \Omega^1 (\log  \mathbb{C})$ with $X$ to be determined, that we denote by $|ij\rangle \dlog (z_i-z_j)$. In this notation if $\tilde{\gamma} \in \mathbb{C} \times I$ is a geometric braid obtained from lifting a loop $\gamma$ in $X_N$, if we arbitrarily index the $N$ strands of $\tilde{\gamma}$, then the $k$-th strand is obtained from lifting a path in the complex plane given by some function $z(t)$, $t \in I$. For a chord $|ij\rangle$ between the $i$-th and the $j$-th strands which are the respective lifts of paths $\gamma_i$ and $\gamma_j$ in the complex plane given by functions $z_i(t)$ and $z_j(t)$, $t\in I$, then $z_i-z_j$ is the difference of two such functions. This leads us to defining the subspace $\Omega^1 (\log \!\vartriangle \! \!\mathbb{C})$ of log differential functionals on $\mathbb{C}$, defined by $\dlog(\vartriangle \!\! z[z_1, z_2])=\dlog(z_1-z_2)$. We have a projection:
\begin{align}
\Omega^1(log \!\vartriangle \!\!\mathbb{C}) & \xrightarrow{p_2} (\mathbb{C}^2-\Delta)/S_2 \\
dlog(z_i-z_j) &\mapsto \{z_i,z_j\}
\end{align}
On the other hand, $|ij\rangle$ represents a chord stretching between the $i$-th and $j$-th strand of a given braid. We define a projection:
\begin{align}
\mathcal{A}^{(1)}(braid) &\xrightarrow{\widetilde{p_1}} (\mathbb{C}^2-\Delta)/S_2 \\
|ij\rangle(Z) &\mapsto \{z_i,z_j\}
\end{align}
It follows that we must have $|ij\rangle \dlog(z_i-z_j) \in \mathcal{A}^{(1)}(\text{braid})\times_{S_2 \mathbb{C}} \Omega^1 (\log \!\vartriangle \!\!\mathbb{C}):=V^{(1)}$, and $|ij\rangle \dlog (z_i-z_j) \in V^{(1)}_{ij}$, so that:
\beq
V^{(1)}=\oplus_{1 \leq i<j \leq N} V^{(1)}_{ij}
\eeq
If necessary we may use a subscript $V_X^{(1)}$ to emphasize the dependency of that vector space on the one dimensional manifold X.\\

Using the $\ess$-product on chord diagrams, extended to chord diagram valued log differential as follows:
\beq
\begin{array}{c}
  |ij\rangle.dlog(z_i-z_j) \\
  \ess\\
  |kl\rangle.dlog(z_k-z_l)
  \end{array}=
  \begin{array}{c}
  |ij\rangle\\
    \ess \\
    |kl\rangle
  \end{array}.dlog(z_i-z_j).dlog(z_k-z_l)
\eeq
if we define:
\beq
V^{(m)}=\ess_{1 \leq k \leq m}V^{(1)}
\eeq
as well as:
\beq
V^{(m)}_{IJ}=\ess_{1 \leq k \leq m}V^{(1)}_{i_k,j_k}
\eeq
where $I=(i_1,\cdots, i_m)$ and $J=(j_1,\cdots, j_m)$, we have the following result:
\begin{prodV}
\beq
V^{(m)}=\bigoplus_{\substack{ I<J \\ I,J \in \{1,...,N\}^m}} V^{(m)}_{I,J}
\eeq
where $I=(i_1,...,i_m)<J=(j_1,...,j_m)$ if $i_k<j_k$ for all k, $\; 1 \leq k \leq m $.
\end{prodV}
\begin{proof}
It suffices to write:
\begin{align}
V^{(m)} &= \ess_{1 \leq k \leq m} V^{(1)}\\
&=\ess_{1 \leq k \leq m} \oplus_{1 \leq i_k < j_k \leq N} V^{(1)}_{i_k,j_k}\\
& \nonumber \\
&=\bigoplus_{\substack{1 \leq i_1<j_1 \leq N \\ \vdots \\ 1 \leq i_m<j_m \leq N}} \begin{array}{c}
V^{(1)}_{i_1,j_1} \\
\ess \\
 \vdots \\
  \ess \\
   V^{(1)}_{i_m,j_m}
    \end{array}\\                                                                             & \nonumber \\
&=\bigoplus_{\substack{I<J \\ I,J \in \{1,...,N\}^m}} V^{(m)}_{I,J}
\end{align}
where $I=(i_1,...,i_m)<J=(j_1,...,j_m)$ if $i_k<j_k$ for all k, $\; 1 \leq k \leq m $.
\end{proof}

 The elements of $V^{(m)}$ are of the form:\\
 \setlength{\unitlength}{0.5cm}
\begin{picture}(3,4)(-6,0)
\multiput(1.2,1.6)(0.2,0){12}{\line(1,0){0.1}}
\multiput(2.4,0.4)(0.2,0){12}{\line(1,0){0.1}}
\multiput(3,0.8)(0,0.2){3}{\circle*{0.05}}
\linethickness{0.3mm}
\multiput(0.4,1)(0.2,0){3}{\circle*{0.05}}
\multiput(5.2,1)(0.2,0){3}{\circle*{0.05}}
\linethickness{0.6mm}
\put(0,0){\line(0,1){2}}
\put(1.2,0){\line(0,1){2}}
\put(2.4,0){\line(0,1){2}}
\put(3.6,0){\line(0,1){2}}
\put(4.8,0){\line(0,1){2}}
\put(6,0){\line(0,1){2}}
\put(0.1,-0.4){$1$}
\put(1.3,-0.4){$i$}
\put(2.5,-0.4){$j$}
\put(3.7,-0.4){$k$}
\put(4.9,-0.4){$l$}
\put(6.1,-0.4){$N$}
\put(6.3,1){$.$}
\put(6.5,1){$\prod_{1 \leq k \leq m}dlog(z_{s_k}-z_{t_k})$}
\end{picture}\\ \\
where in this example we have:
\begin{align}
s_1=j &, t_1=l\\
&\vdots \nonumber \\
s_m=i &, t_m=k
\end{align}
and $z_{s_i},z_{t_i}$ are the coordinates in the complex plane for the points of intersection  of the $s_i$-th and $t_i$-th strands respectively with the complex plane containing the $i$-th chord. We can write such elements in compact form as $\sum_{|P|=m} \lambda_P . |P\rangle .\text{d}^m \log \! \vartriangle \!\!z[P]$, the $\lambda_P$ are complex coefficients, and $\text{d}^m \log\! \vartriangle \!\!z[P]$ is short for $\prod_{1 \leq i \leq m} \dlog(\vartriangle \!\!z[P_i])=\prod_{1 \leq k \leq m} \dlog (z_{s_k}-z_{t_k})$ where we define $\vartriangle \!\!z[P_i]=z_k-z_l$ if $\widetilde{p_1}|P_i\rangle(Z) = \{z_{s_i}, z_{t_i}\}$.\\

\section{Background material}
It is possible to generalize the work of Chen ~\cite{ChI},~\cite{ChII} on connections and holonomy of non-commutative graded algebras fibered over manifolds to the case where the algebra is taken to be $\overline{\mathcal{A}}(I^N)$ where $I^N$ denotes the trivial braid with $N$ vertical strands, equally spaced, of length one, leading to express the Kontsevich integral as the transport of a formal connection along a loop in $C_N \mathbb{C}=\mathbb{C}^N - \Delta$ ~\cite{Koh},~\cite{DBN}. We briefly summarize how this is done.\\

Chen ~\cite{ChI} first extended the definition of a linear connection on a vector bundle over some manifold to the case where the bundle has as fiber a not necessarily finite dimensional, non-commutative graded algebra. Consider $(\overline{\mathcal{A}}(I^N),+,\ess)$ ~\cite{DBN} which in this simple case has a product $\ess$ given by the connected sum $\#$. For two chord diagrams $|P\rangle(I^N)$ and $|P'\rangle(I^N)$ with support on $I^N$, one can define $|P\rangle(I^N)\# |P'\rangle(I^N)$ by gluing $|P\rangle(I^N)(t=0)$ to $|P'\rangle(I^N)(t=1)$ strand-wise. One further considers such chord diagrams mod the 4-T and framing independence relations, thus making the algebra with $\#$ as a product a non-commutative algebra. We see this algebra as being trivially fibered over $C_N \mathbb{C}$, and we put some connection $\nabla=d-\omega$ on it ~\cite{ChII}. One chooses $\omega$ to be the usual Knizhnik-Zamolodchikov connection 1-form:
 \beq
 \omega=\frac{1}{2\pi i}\sum_{1 \leq k < l \leq N}|kl\rangle(I^N).\dlog \vartriangle\!\!z_{kl}
 \eeq
 The transport of a section $\xi$ of $\overline{\mathcal{A}}(I^N)$ is given by solving the following graded equation for the parallel transport $\tau$ of $\xi$ along a loop $\gamma$ in the base $\CNC$ ~\cite{ChI}:
 \beq
 d\tau=\omega \tau
 \eeq
Solving this formally degree by degree, one gets the transport $\tau$ of the connection $\omega$ ~\cite{ChII}:
\beq
\tau=1+\sum_{n > 0}\int \omega^n
\eeq
Formally:
\begin{align}
\tau &=1 + \sum_{k<l}\left(\frac{1}{2 \pi i}\int \dlog \vartriangle \!\! z_{kl}\right)|kl\rangle (I^N) \nonumber \\
&+\sum_{\substack{i<j \\ k<l}}\left(\frac{1}{(2 \pi i)^2}\int_{t_1 < t_2}\dlog \vartriangle \!\!z_{ij}(t_2) \dlog \vartriangle \!\! z_{kl}(t_1)\right)|ij\rangle (I^N) \# |kl \rangle (I^N) + \cdots
\end{align}
Define a pairing:
\begin{align}
\langle - , - \rangle:T^*\CNC \times C_*(L \CNC) & \rightarrow \mathbb{C} \nonumber \\
(\omega, c) & \mapsto \langle \omega , c \rangle :=\int_c \omega = \int c^* \omega
\end{align}
for some smooth $c: S^1 \rightarrow \CNC$, and $C_*(L \CNC)$ denotes the free chain complex of non-degenerate loops in $\CNC$. Armed with these notations one can write, for $\gamma \in L \CNC$:
\begin{align}
\langle \tau, \gamma \rangle &=\langle 1, \gamma \rangle + \sum_{n > 0}\langle \omega^n, \gamma \rangle \\
&=\langle 1, \gamma \rangle + \sum_{k<l}\left(\frac{1}{2 \pi i}\int_{\gamma} \dlog \vartriangle \!\! z_{kl}\right)|kl\rangle (I^N) \nonumber \\
&+\sum_{\substack{i<j \\ k<l}}\left(\frac{1}{(2 \pi i)^2}\int_{\gamma ,\; t_1 < t_2}\dlog \vartriangle \!\!z_{ij}(t_2) \dlog \vartriangle \!\! z_{kl}(t_1)\right)|ij\rangle (I^N) \# |kl \rangle (I^N) + \cdots
\end{align}

which can equivalently be obtained by solving the Knizhnik-Zamolodchikov equation ~\cite{Koh} by the method of iterated integrals. Thus one obtains the generalized holonomy map ~\cite{ChII}:
\begin{align}
\Theta:\Omega C_N \mathbb{C} &\rightarrow \overline{\mathcal{A}}(I^N) \nonumber \\
&\gamma \mapsto \langle \tau,\gamma \rangle=\sum_{n \geq 0}\int_{\gamma} \omega^n
\end{align}
which is a sum of chord diagrams with support $I^N$ with coefficients of the form $(1/(2 \pi i)^m)\int_{\gamma}\text{d}^m \log \vartriangle\!\!z[P](\gamma)$ for a degree $m$ chord diagram. For some loop $\gamma$,  $\sum_{n \geq 0}\int_{\gamma} \omega^n$ is also referred to as the holonomy of the connection $\omega$ along the loop $\gamma$. It was argued ~\cite{Koh} that Kontsevich generalized this in his seminal paper to the Kontsevich integral of a knot, replacing trivial chord diagrams on $I^N$ by chord diagrams on $S^1$, starting from the actual knot $\gamma$.

\section{The Kontsevich integral as a holonomy}
In this section, we give a direct construction of the Kontsevich integral of a knot as being generated by the holonomy of some bundle of chord diagrams and not as a generalization of the formal holonomy of the Knizhnik-Zamolodchikov connection on $\barA(I^N)$. Braids are more naturally defined in $X_N:=(\mathbb{C}^N-\Delta)/S_N=C_N \mathbb{C}/S_N$ and not $C_N \mathbb{C}$ and it is on this space that we will consider the algebra bundle $(\overline{\mathcal{A}}(X_N),+,\cdot)$ of chord diagrams based at points of $X_N$, for which the product $\ess$ is simply denoted by $\cdot$ as it is a commutative product in this case. Defining the bundle to be $\overline{\mathcal{A}}(X_N)$ instead of $\overline{\mathcal{A}}(I^N)$ will allow us to have integrands that are not simply chord diagrams with support on $I^N$ but chord diagrams with support the link itself, the way Kontsevich initially wrote down his formula for his integral. In this formalism, for a link $L$ that once put in braid position corresponds to a geometric braid $\tilde{\gamma} \in \mathbb{C} \times I$ with $N$ strands, obtained by lifting some loop $\gamma \in X_N$, then $Z(L)$ is generated by the holonomy for the Knizhnik-Zamolodchikov connection around the loop $\gamma$ in the base. It is important to make a mental distinction between the algebra $\barAXN$ that has chords based at different points, and the bundle $\barAXN \rightarrow X_N$ whose fibers are chord diagrams based at a same point. We will strive to always make this distinction evident by saying whether we work with the bundle or the algebra itself.\\

To achieve this, we first define the bundle $\overline{\mathcal{A}}(C_N \mathbb{C})$ over $C_N \mathbb{C}$ on which we put the usual Knizhnik-Zamolodchikov connection, define an action of the symmetric group on sections of this bundle leading to the definition of the algebra bundle $C_N \mathbb{C} \times_{S_N \mathbb{C}} \overline{\mathcal{A}}(C_N \mathbb{C}) / S_N$ over $X_N$ with total space $\overline{\mathcal{A}}(X_N)$. The KZ connection on $\overline{\mathcal{A}}(C_N \mathbb{C})$ descends to a well-defined connection on this bundle. Solving the KZ equation exactly, we obtain the holonomy for that connection around a loop $\gamma$ in the base $X_N$. Upon expansion of this holonomy we generate $Z(L)$.\\

\subsection{$\overline{\mathcal{A}}(C_N \mathbb{C})$ as a bundle over $C_N \mathbb{C}$}
We put the product topology on $\mathbb{C}^N$, and the subspace topology on $C_N \mathbb{C}=\mathbb{C}^N-\Delta$. We consider over this space the algebra bundle $(\barACNC,+,\cdot)$ where $\mathcal{A}(C_N \mathbb{C})$ is the $\mathbb{C}$-linear span of chord diagrams with support a same point of $C_N \mathbb{C}$, viewed as a graded algebra over $\mathbb{C}$ with graded completion $\barACNC$. We now put a topology on $\barACNC$. By definition, we have this bundle to be trivial, and its fiber is isomorphic as an algebra to the algebra of formal power series in $N(N-1)/2$ variables:
\beq
\barACNC \simeq C_N \mathbb{C} \times \mathbb{C}[[\{|ij\rangle \;|\; 1 \leq i<j \leq N\}]]
\eeq

We put on this total space the topology defined as the product of the topology on $C_N \mathbb{C}$ with a topology on the algebra $\mathbb{C}[[\{|ij\rangle \;|\; 1 \leq i<j \leq N\}]]$ as we do now. We first start by considering the ring of formal power series $\mathbb{C}[[x]]$ whose elements are of the form $\sum_{n \geq 0} a_n x^n$ and are determined by giving a sequence $(a_n) \in \mathbb{C}^{\omega}$. We define the following metric on this ring; if $f(x)=\sum_{n \geq 0} a_n x^n$ and $g(x)=\sum_{n \geq 0}b_n x^n$, then we consider the metric $d$ induced by the distance function on sequences $(a_n)$ and $(b_n)$ defined by:
\beq
d((a_n),(b_n))=2^{-k}
\eeq
where $k$ is the smallest integer for which $a_n \neq b_n$. Then $f$ and $g$ being defined by the respective sequences $(a_n)$ and $(b_n)$, we define $d(f,g)=2^{-k}$. To show that this defines a metric on the formal power series ring is not difficult, the only tricky point being to show the triangle inequality. If $h(x)=\sum_{n \geq 0}c_n x^n$ is a third power series such that $d((a_n),(c_n))=2^{-i}$ and $d((b_n),(c_n))=2^{-j}$ for some integers $i$ and $j$, then we would like to have $d((a_n),(b_n)) \leq d((a_n),(c_n)) + d((b_n),(c_n))$, that is $2^{-k} \leq 2^{-i} + 2^{-j}$. The only combinations of indices $i,j,k$ that does not lead to this inequality are the two cases $k<i<j$ and $k<j<i$, none of which occur. Indeed if either case is true, then in particular we could have $a_k=c_k$ and $b_k=c_k$, a contradiction with the definition of $k$. With this metric we define a metric topology on the ring of formal power series, thereby making it into a topological space.\\

We now deal with the ring of formal power series in two commuting unknowns $x$ and $y$ denoted by $\mathbb{C}[[x,y]]$ whose elements are of the form $\sum_{m \geq 0} \sum_{m_1 + m_2=m} a_{m_1,m_2}\,x^{m_1}y^{m_2}$. We put on this ring the metric topology with the use of the metric induced by the distance function:
\beq
d((a_{m_1,m_2}),(b_{m_1,m_2}))=2^{-k}
\eeq
where:
\beq
k=\text{inf}\{m_1+m_2 \;|\;a_{m_1,m_2} \neq b_{m_1,m_2}\}
\eeq
As in the case of one variable only, it is clear that this induces a metric on the ring of formal series, and the only difficulty in showing that this does define a metric is the triangle inequality. If we consider a third series determined by a sequence $(c_{m_1,m_2})$ and we have $d((a_{m_1,m_2}),(c_{m_1.m_2}))=2^{-i}$, $d((b_{m_1,m_2}),(c_{m_1.m_2}))=2^{-j}$, then the only combinations of indices that raise problems are the two cases $k<i<j$ and $k<j<i$. For either case, if $m_1$ and $m_2$ are two integers whose sum is equal to $k$, then we could have equalities $a_{m_1,m_2}=c_{m_1,m_2}$ and $b_{m_1,m_2}=c_{m_1,m_2}$, in contradiction with the definition of $k$.\\

It is clear that $\mathbb{C}[[\{|ij\rangle \;|\; 1 \leq i<j \leq N\}]]$, whose elements are of the form:
\beq
\sum_{m \geq 0} \sum_{\sum_{1 \leq i<j \leq N}m_{ij}=m} a_{(m_{ij})}\,\prod_{i<j}(|ij\rangle)^{m_{ij}}
\eeq
can be topologized with the metric induced by the distance function:
\beq
d((a_{(m_{ij})}),(b_{(m_{ij})}))=2^{-k}
\eeq
where:
\beq
k=\text{inf}\{\sum_{1 \leq i<j \leq N} m_{ij} \;|\; a_{(m_{ij})} \neq b_{(m_{ij})}\}
\eeq
\\

\subsection{Connection on $\barACNC$}
Before putting a connection on $\barACNC \rightarrow C_N \mathbb{C}$, we briefly review how a connection is defined on a vector bundle. From first principles, let $E:X\times \mathbb{C}^n$ be an $n$-dimensional vector bundle over some topological space $X$, $\Gamma(E) \simeq \mathcal{C}^{\infty}(X, \mathbb{C}^n)$ its space of sections, $\sigma \in \Gamma(E)$, then $d\sigma$ can be computed using multivariate calculus. For $f$ a function on $X$, we have $d(f.\sigma)=df\otimes\sigma + f.d\sigma$. If $E$ is not trivial however, we have to generalize this notion of differentiation to a $\mathbb{C}$-linear differential operator $\nabla:\Gamma(E) \rightarrow \Gamma(T^*X \otimes E)$ satisfying the Leibniz rule $\nabla(f \sigma)=df \otimes \sigma + f \nabla \sigma$. If $E \stackrel{\pi}{\rightarrow} X$ is a smooth complex vector bundle with local trivializations $E|_U \stackrel{\sim}{\rightarrow} U \times \mathbb{C}^n$ given by a local frame field $e_1,...,e_n$ on U, then any $\sigma \in \Gamma(E|_U)$ can be written locally:
\beq
\sigma(x)=\sum_{1 \leq k \leq n}a_k(x)e_k(x) \simeq (a_1(x),...,a_n(x)) \equiv a(x)
\eeq
for $x \in U$. If there is a connection $\nabla$ on $E$, then we write:
\beq
\nabla e_j=\sum_{1 \leq k \leq n}\omega_{kj}\otimes e_k
\eeq
and we let $\omega:=(\omega_{kj})_{1 \leq k,j \leq n}$. For $\sigma \in \Gamma(E|_U)$ as above, $\sigma \simeq a$, then we have:
\begin{align}
\nabla \sigma&= \nabla(\sum_j a_j e_j) \\
&=\sum_j\nabla(a_je_j) \\
&=\underbrace{\sum_j(da_j)\otimes e_j}_{da} + \sum_j a_j\!\!\!\!\!\underbrace{\nabla e_j}_{\sum_{1 \leq k \leq n}\omega_{kj}\otimes e_k} \\
& \nonumber \\
&=da + \sum_{jk} a_j \omega_{kj} e_k \\
&=da + \underbrace{\sum_k (\omega a )_k e_k}_{\omega a} \\
&=(d+\omega)a
\end{align}
Chen ~\cite{ChI} generalized this to the case of $T^*X \otimes A$ over some base space $X$, where $T^*X$ is the cotangent bundle of $X$, $A$ is some graded, non-commutative algebra. For our purposes, $A=\overline{\mathcal{A}}(C_N \mathbb{C})$, and we put on $\overline{\mathcal{A}}(C_N \mathbb{C})$ the Knizhnik-Zamolodchikov connection $\nabla=d-\omega$ where:
\beq
\omega=\frac{1}{2 \pi i}\sum_{1 \leq i\neq j \leq N}\frac{|ij\rangle}{z_i-z_j}dz_i
\eeq
the well-known KZ connection form. This gives a well-defined connection on $\overline{\mathcal{A}}(C_N \mathbb{C})$; it is immediate that both linearity and the Leibniz rule hold for this choice of connection.\\ \\

\subsection{Construction of the bundle $\barAXN$ over $X_N$} \label{constr}
We first define an action of the permutation group $S_N$ on sections of $\barACNC$. Let $Z \in \CNC$ and $\psi \in \Gamma(\barACNC)$. For some element $\sigma \in S_N$, we would like $(\sigma \psi)(\sigma Z)=\psi (Z)$ for consistency, where $\sigma Z=\sigma (z_1,...,z_N)=(z_{\sigma(1)},...,z_{\sigma(N)})$. Write $\psi(Z) = \sum_P \lambda_P |P\rangle(Z)$, which is short for $(Z,\sum_P \lambda_P |P\rangle)$, where the sum is over all chord diagrams $|P\rangle$ with support on $Z$, and $\lambda_P \in \mathbb{C}$ for all $P$. Then we must have $\sigma \psi (\sigma Z)=\sum_P \lambda_P |\sigma P\rangle ( \sigma Z)$, where if $P=(P_1,...,P_m)$, then we define $\sigma P:=(\sigma P_1,...,\sigma P_m)$, with $\sigma |kl\rangle=|\sigma k \sigma l\rangle$, $1 \leq k,l \leq N$. Thus we define:
\beq
\sigma \psi = \sum_P \lambda _P |\sigma P \rangle
\eeq
Consider the set of equivalence classes $\CNC \times_{S_N \mathbb{C}} \barACNC / S_N$, whose elements are classes $[(Z, \psi(Z))]=\{(\sigma Z,\sigma \psi (\sigma Z)) \;|\;\sigma \in S_N, \;Z \in C_N \mathbb{C},\; \psi(Z) \in \barA(Z)\}$. This defines an algebra bundle by projecting on the second factor:
\beq
\begin{CD}
 \barAXN \\
 @V \pi VV \\
  X_N
  \end{CD}
 \eeq
 \\

We now give a description of this bundle. Fix some $Z \in X_N$. If $q$ is the quotient map from $C_N \mathbb{C}$ to $X_N$, then for some $\tilde{Z} \in q^{-1}(Z)$, a choice of some element $\sum \lambda_P |P\rangle (\tilde{Z}) \in \overline{\mathcal{A}}(\tilde{Z})$ together with the point $\tilde{Z}$ defines a representative $(\tilde{Z}, \sum \lambda_P |P\rangle (\tilde{Z}))$ of some class in $C_N \mathbb{C} \times_{S_N \mathbb{C}} \barACNC / S_N$. Now $|P\rangle (\tilde{Z}) \in \mathcal{A}(C_N \mathbb{C})$ only. However projecting $\tilde{Z}$ to $Z$ those two coordinates of $\tilde{Z}$ that are marked by a chord by $|P\rangle$ are also marked in the complex plane after projection, and thus defines a chord diagram that we still denote by $|P\rangle(Z)$ for convenience. Note however that if $|P\rangle (\tilde{Z})$ can be defined using indices, $|P\rangle (Z)$ is index-free. Further the definition of $|P\rangle (Z)$ is independent of the representative chosen.

\begin{labeling}
Fix N=3, $(e_i)_{1 \leq i \leq 3}$ the canonical basis for $\mathbb{C}^3$, $\tilde{Z}=(z_1,z_2,z_3) \in C_3 \mathbb{C}$, $|12\rangle(\tilde{Z})$ projects to the chord diagram $|Q\rangle$ at $Z=q \tilde{Z} \in X_3$:\\
\setlength{\unitlength}{0.5cm}
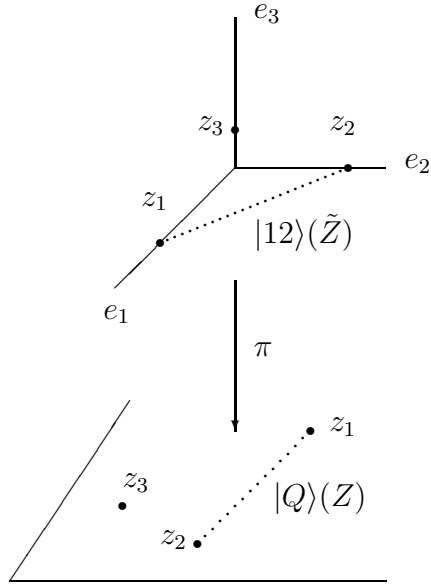
\begin{picture}(12,16)(-9,0)
\put(0,0){\line(2,3){3.2}}
\put(0,0){\line(1,0){10}}
\put(3,2){\circle*{0.2}}
\put(5,1){\circle*{0.2}}
\put(8,4){\circle*{0.2}}
\put(3,2.5){$z_3$}
\put(4,1){$z_2$}
\put(8.5,4){$z_1$}
\multiput(5,1)(0.2,0.2){16}{\circle*{0.1}}
\put(7,2){$|Q\rangle(Z)$}
\put(6,8){\line(0,-1){4}}
\put(6,5){\vector(0,-1){1}}
\put(6.5,6){$\pi$}
\put(6,11){\line(1,0){4}}
\put(6,11){\line(0,1){4}}
\put(6,11){\line(-1,-1){3.2}}
\put(10.5,11){$e_2$}
\put(2.5,7){$e_1$}
\put(6.5,15){$e_3$}
\put(4,9){\circle*{0.2}}
\put(9,11){\circle*{0.2}}
\put(6,12){\circle*{0.2}}
\put(3.5,10){$z_1$}
\put(8.5,12){$z_2$}
\put(5,12){$z_3$}
\put(6.5,9){$|12\rangle(\tilde{Z})$}
\multiput(4,9)(0.2,0.08){26}{\circle*{0.1}}
\end{picture}\\ \\
with $Z=\{z_1,z_2,z_3\} \in X_3$, and the chord $|Q\rangle$ along with the point in $X_3$ are represented by:\\
\setlength{\unitlength}{0.5cm}
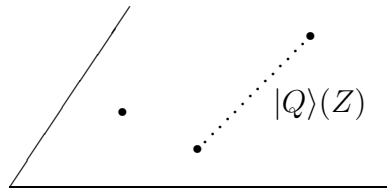
\begin{picture}(12,6)(-9,0)
\put(0,0){\line(2,3){3.2}}
\put(0,0){\line(1,0){10}}
\put(3,2){\circle*{0.2}}
\put(5,1){\circle*{0.2}}
\put(8,4){\circle*{0.2}}
\multiput(5,1)(0.2,0.2){16}{\circle*{0.1}}
\put(7,2){$|Q\rangle(Z)$}
\end{picture}\\ \\
Conversely, given $Z$ as above along with the chord $|Q\rangle$, pick a labeling of the 3 defining points of $Z$:\\
\setlength{\unitlength}{0.5cm}
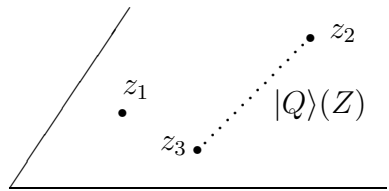
\begin{picture}(12,6)(-9,0)
\put(0,0){\line(2,3){3.2}}
\put(0,0){\line(1,0){10}}
\put(3,2){\circle*{0.2}}
\put(5,1){\circle*{0.2}}
\put(8,4){\circle*{0.2}}
\put(3,2.5){$z_1$}
\put(4,1){$z_3$}
\put(8.5,4){$z_2$}
\multiput(5,1)(0.2,0.2){16}{\circle*{0.1}}
\put(7,2){$|Q\rangle(Z)$}
\end{picture}\\ \\
Then for $\tilde{Z}=(z_1,z_2,z_3) \in q^{-1}Z$, we have $|Q\rangle (\tilde{Z})=|23\rangle (\tilde{Z})$.
\end{labeling}

Note that $\sum_Q \lambda_Q |Q\rangle \in \mathbb{C}[[\{|P\rangle \}]]$, the power series algebra on $N(N-1)/2$ generators which are the $N(N-1)/2$ possible chords of order 1 on a given point of $X_N$. Thus the fiber of $\barAXN$ over $X_N$ is isomorphic to $\mathbb{C}[[\{|P\rangle \;|\;|P|=1\}]]$ and inherits its automorphism group as structure group: $Aut(\barAXN)=Aut(\mathbb{C}[[\{|P\rangle \}]])$.\\

\subsection{KZ connection on $\barAXN$}
Recall that we had the KZ connection $\nabla =d-\omega$ on $\barACNC$ over $\CNC$, with:
\beq
\omega=\frac{1}{2 \pi i}\sum_{1 \leq i \neq j \leq N}\frac{|ij\rangle}{z_i-z_j}dz_i
\eeq
We show it is invariant under the action of the symmetric group $S_N$, and therefore descends to a connection on $\barAXN$:
\begin{desc}
The KZ connection on $\barACNC$ over $\CNC$ descends to a well defined connection $\nabla=d-\omega$ on $\barAXN$ over $X_N$, with:
\beq
\omega=\frac{1}{2 \pi i}\sum_{|P|=1}|P\rangle.\dlogp
\eeq
\end{desc}
\begin{proof}
It suffices to observe that we can write:
\begin{align}
\sum_{i \neq j} \frac{|ij\rangle}{z_i-z_j}dz_i &=\sum_i (\sum_{j>i}\frac{|ij\rangle}{z_i-z_j}dz_i + \sum_{j<i} \frac{|ij\rangle}{z_i-z_j}dz_i) \\
&=\sum_{i<j}\frac{|ij\rangle}{z_i-z_j}dz_i - \sum_{j<i}\frac{|ji\rangle}{z_j-z_i}dz_i \label{pr1}\\
&=\sum_{i<j}\frac{|ij\rangle}{z_i-z_j}dz_i-\sum_{i<j}\frac{|ij\rangle}{z_i-z_j}dz_j \label{pr2}\\
&=\sum_{i<j}\frac{|ij\rangle}{z_i-z_j}(dz_i-dz_j) \\
&=\sum_{1 \leq i<j \leq N}|ij\rangle.\dlog(z_i-z_j)
\end{align}
where in going from ~\eqref{pr1} to ~\eqref{pr2} we used the fact that for all indices $i \neq j$, we have $|ij\rangle=|ji\rangle$. From this we can write the KZ connection $\nabla=d-\omega$ on $\barACNC$ over $\CNC$ as $\omega=(1/2 \pi i)\sum_{|P|=1}|P\rangle.\dlogp$, which is invariant under the action of $S_N$.
It follows that the KZ connection on $\barACNC$ descends to a well-defined connection $\nabla=d - \omega$ on $\barAXN$ over $X_N$, with $\omega=(1/2 \pi i)\sum_{|P|=1}|P\rangle.\dlog \vartriangle\!\! z[P]$.
\end{proof}

\subsection{subtleties of $TX_N$ and $T^*X_N$}
If $\omega$ as defined in the previous section unambiguously makes sense by virtue of the presence of chords in its definition, the tangent and cotangent bundles of $X_N$ are not well-defined however. Points in $X_N$ can be written $\sum_{1 \leq i \leq N}[z_i]$ where the indexing $[z_i]$ on the $N$ defining points of $Z$ is nothing but a choice of labeling. If nearby points can be chosen to have the same labeling, this is not necessarily true of points far away from $Z$. To remedy this we endow connected open sets $U=\{U_1, \cdots, U_N \}$ - all $U_i$'s being connected - with a property that we refer to as the "small enough" property. This means that points in small enough connected open sets have their $N$ defining points with a same choice of labeling, and points in disjoint connected open sets may not necessarily have defining points with a same choice of labeling. Both the tangent and the cotangent bundles of $X_N$ have local trivializations defined over small enough, connected open sets, for in this case each defining open set for such small enough connected open set is connected, and complex variables carry a same choice of labeling, whence the possibility of defining local complex variables in those. From there, tangent vectors can be locally defined, as well as their duals. It is in this sense that $TX_N$ and $T^*X_N$ can be defined.

\subsection{The KZ equation}
Let $\xi$ be a section of $\barAXN$:
\beq
\xi=\sum_P \lambda_P|P\rangle
\eeq
Let $Z \in X_N$, $\gamma$ a loop in $X_N$ based at $Z$. The parallel transport of $\xi$ along $\gamma$ is denoted by $\Psi$. It is a horizontal section for the connection, so that $\nabla \Psi=0$ from which it follows that:
\begin{align}
d\Psi &=\omega \Psi \\
&=\frac{1}{2 \pi i}\sum_{|P|=1} |P\rangle.\dlogp \cdot \Psi
\end{align}
and using the fact that the fiber of $\barAXN$ is isomorphic to $\mathbb{C}[[\{|P\rangle\}]]$, a commutative algebra, this equation has the following solution:
\begin{align}
\Psi &=e^{(1/2 \pi i)\sum_P |P\rangle.\log \vartriangle z[P] + \Theta} \\
&=e^{\Theta}\cdot e^{(1/2 \pi i)\sum_P |P\rangle.\log \vartriangle z[P]} \\
&=\mu e^{(1/2 \pi i)\sum_P |P\rangle.\log \vartriangle z[P]}
\end{align}
where $\Theta \in \barA$ is a constant term determined by the initial condition $
\xi(Z)=\Psi(Z)$ and we have put $\mu \equiv \exp(\Theta)$. Solving for the initial condition:
\beq
\Psi(Z)=\mu e^{(1/2 \pi i)\sum_P |P\rangle.\log \vartriangle z[P](Z)}=\xi(Z)
\eeq
using the notation:
\beq
\Lambda \equiv \frac{1}{2 \pi i}\sum_{|P|=1} |P\rangle .\log \vartriangle \!\!z[P]
\eeq
then we must have:
\beq
\mu=e^{-\Lambda(Z)} \cdot \xi(Z)
\eeq

\subsection{Holonomy of the KZ equation}
With the same notations and definitions as in the previous subsection, we have:
\begin{align}
\Psi(\gamma(1))&=\mu.e^{\Lambda(\gamma(1))} \\
&=\mu.e^{\int_I \omega(\gamma) + \Lambda(\gamma(0))} \\
&=e^{\int_I \omega(\gamma)}\cdot \mu \cdot e^{\Lambda(\gamma(0))} \\
&=e^{\int_I \omega (\gamma)}.\Psi(\gamma(0))
\end{align}
and the holonomy of the KZ connection along a loop $\gamma$ in the base is expressed as $\exp (\int_I \omega(\gamma))$.
\begin{holo}
For a link $L$ that once put in braid position is given by a geometric braid $\tilde{\gamma} \in \mathbb{C}\times I$ corresponding to the lift of some loop $\gamma \in X_N$, the holonomy $h$ of the KZ connection on $\barAXN$ over $XN$ along $\gamma$ generates $Z(L)$, the Kontsevich integral for the link $L$.
\end{holo}
\begin{proof}
It suffices to write:
\begin{align}
h \equiv e^{\int_I \omega(\gamma)} &=\sum_{m \geq 0}\frac{1}{m!}(\int_I \omega)^m \\
&=\sum_{m \geq 0} \frac{1}{m!}\int_{I^m} \omega^m
\end{align}
with:
\beq
\omega^m(t_1,...,t_m)=\frac{1}{(2 \pi i)^m}\sum_{P_1,...,P_m}\prod_{1 \leq i \leq m}|P_i\rangle(\gamma t_i)\dlog \vartriangle \!\!z[P_i](\gamma t_i) \label{alg}
\eeq
It is worthwile to point out that though each $|P_i\rangle (\gamma t_i)$ is an element of the fiber of $\barAXN$ over the point $\gamma (t_i)$, $1 \leq i \leq m$, the product of such elements as in ~\eqref{alg} is no longer an element of a fiber of $\barAXN$ but is truly an element of the algebra $\barAXN$ proper. We can write the holonomy $h$ as:
\begin{align}
h&=\sum_{m \geq 0} \frac{1}{m!}\frac{1}{(2 \pi i)^m}\sum_{P_1,...,P_m} \int_{I^m}\prod_{1 \leq i \leq m}|P_i\rangle(\gamma t_i)\dlog \vartriangle \!\!z[P_i](\gamma t_i)\\
&=\sum_{m \geq 0} \frac{1}{m!}\frac{1}{(2 \pi i)^m}\sum_{P_1,...,P_m} \sum_{\sigma \in S_m} \int_{\Delta_m \ni (t_1,...,t_m )} \prod_{1 \leq i \leq m}|P_{\sigma(i)}\rangle(\gamma t_i)\dlog \vartriangle \!\!z[P_{\sigma(i)}](\gamma t_i) \label{delm}\\
&=\sum_{m \geq 0} \frac{1}{m!}  \frac{1}{(2 \pi i)^m}\sum_{\sigma \in S_m} \int_{\Delta_m \ni (t_1,...,t_m )}\sum_{P_1,...,P_m} \prod_{1 \leq i \leq m}|P_{\sigma(i)}\rangle(\gamma t_i)\dlog \vartriangle \!\!z[P_{\sigma(i)}](\gamma t_i) \\
&=\sum_{m \geq 0} \frac{1}{m!} \frac{1}{(2 \pi i)^m} \sum_{\sigma \in S_m} \int_{\Delta_m \ni (t_1,...,t_m )}\sum_{P_1,...,P_m} \prod_{1 \leq i \leq m}|P_i\rangle(\gamma t_i)\dlog \vartriangle \!\!z[P_i](\gamma t_i) \\
&=\sum_{m \geq 0} \frac{1}{m!}  \frac{1}{(2 \pi i)^m} m!\int_{\Delta_m \ni (t_1,...,t_m )}\sum_{P_1,...,P_m} \prod_{1 \leq i \leq m}|P_i\rangle(\gamma t_i)\dlog \vartriangle \!\!z[P_i](\gamma t_i)
\end{align}
where in ~\eqref{delm} we have used the notation $\Delta_m=\{0 \leq t_1 < \cdots < t_m \leq 1 \}$. Observe that we have obtained the formal sum of elements of $\overline{\mathcal{A}}(X_N)$, based at points along $\gamma \in \Omega X_N$. In other terms we have all possible chords based at points of $\gamma$. This is sufficient to lift $h$ to $\tilde{h}$, which is the same as $h$ except that its tangle chord diagrams are of the form $|P\rangle (\tilde{\gamma}t_i)$. In doing so, crossings in $X_N$ may lift to positive or negative crossings, which are immaterial once we close the braid $\tilde{\gamma}$ into $L$. Further, from the collection of chord supports, we recover $\tilde{\gamma}$, as the chords trace out sheets whose boundaries are the $N$ strands of $\tilde{\gamma}$.\\

If we write $P=(P_1,...,P_m)$ and $T=(t_1, \cdots, t_m)$, then we can write:
\beq
\prod_{1 \leq i \leq m} |P_i\rangle (\tilde{\gamma}t_i) =
|P\rangle (\tilde{\gamma})(T) \in \mathcal{A}(\tilde{\gamma})
\eeq
Then:
\begin{align}
\tilde{h}&=\sum_{ m\geq 0} \sum_{|P|=m}\frac{1}{(2 \pi i)^m}\int_{\Delta_m \ni T}|P\rangle(\tilde{\gamma})(T)\text{d}^m \log \vartriangle\!\!z[P](\tilde{\gamma})(T) \\
&=\sum_{ m\geq 0} \sum_{|P|=m}\frac{1}{(2 \pi i)^m}\int_{\tilde{\gamma},\Delta_m \ni T}|P\rangle(T)\text{d}^m \log \vartriangle\!\!z[P](T)
\end{align}
where in the second step we have omitted the mention of $\tilde{\gamma}$ in the integrand for notation's sake and have instead put it under the integral sign to emphasize that we integrate along $\tilde{\gamma}$. Construct the closure $\overline{\tilde{\gamma}}$ of the geometric braid $\tilde{\gamma}$ as follows: if the strands of $\tilde{\gamma}$ originate and end at $\{(n,0,0)\;|\;0 \leq n \leq N-1 \}$ and $\{(n,0,1)\;|\; 0 \leq n \leq N-1\}$ respectively in $\mathbb{R}^3$, then glue to these endpoints the parallel strands $\{\{n\}\times I \times \{0\}\;|\; 0 \leq n \leq N-1 \}$ at time $t=0$ and $\{\{n\}\times I \times \{1\}\;|\; 0 \leq n \leq N-1 \}$ at time $t=1$. Finally we glue in the back $N$ parallel strands $\{\{n\}\times \{1\}\times I\;|\; 0 \leq n \leq N-1 \}$ and smooth corners to get a Morse knot. In this manner, we recover the Kontsevich integral as initially defined by Kontsevich \cite{K}, since after smoothing corners $\overline{\tilde{\gamma}}$ yields $L$. We denote the resulting object by $\overline{\tilde{h}}$.
\begin{align}
\overline{\tilde{h}}&=\sum_{ m\geq 0} \sum_{|P|=m}\frac{1}{(2 \pi i)^m}\int_{\overline{\tilde{\gamma}},\Delta_m \ni T}|P\rangle(T)\text{d}^m \log \vartriangle\!\!z[P](T)\\
&=\sum_{m \geq 0} \sum_{|P|=m} \frac{1}{(2 \pi i)^m}\int_{ \Delta_m \ni T} L_P \text{d}^m \log \vartriangle\!\!z[P](T)
\end{align}
the form of the Kontsevich integral of links that is expressed as a sum of tangle chord diagrams and is mostly used to take advantage of the multiplicative property of $Z$.
\end{proof}
In this integral, there are additional chords between strands in the back of the closure, whose strands are all parallel and therefore evaluate to zero. There are also chords between strands of the braid and strands from the back of the closure. Such chords are called long chords and the integrals corresponding to such chords evaluate to zero \cite{ChDu}. However, chords from the top and bottom portions of the closure have a non-zero contribution. They can easily be calculated and their contribution be grafted onto $\tilde{h}$ to yield $\overline{\tilde{h}}$. It is not true however to say that we can divide out these contributions from $Z(L)$ to obtain $\tilde{h}$. This is why we say that the Kontsevich integral is generated by a holonomy rather than saying that the Kontsevich integral is a holonomy in the classical sense.\\

We would like now to relate $\overline{\tilde{h}}$ to $Z(L)$ as an element of $\barA(\coprod S^1)$.  Different tangle chord diagrams of the same degree supported on the same link $L$ will be defined to be equivalent if one can go from one pairing to the other by circling $L$. This is a well-defined equivalence relation. If $P$ is a representative of such a class, we denote the class by $[P]$. \\

Now for $m$ fixed, $P=(P_1,...,P_m)$ fixed, from $L_P$ we can determine the chord diagram on $\coprod S^1$ it corresponds to under the embedding $\iota_L$ defining the link $L$: $\iota_L^{-1} L_P=D_{[P]}$, thereby defining a moduli map:
\beq
\tau: \overline{\mathcal{A}}(\mathcal{L}inks) \rightarrow \overline{\mathcal{A}}(\amalg S^1)
\eeq
that we extend by linearity. Then:
\begin{align}
\tau \overline{\tilde{h}}&=\tau \sum_{m \geq 0} \sum_{|P|=m} \frac{1}{(2 \pi i)^m} \int_{L, \Delta_m \ni T} L_P \text{d}^m \log \vartriangle \!\!z[P](T) \\
&=\sum_{m \geq 0} \sum_{|P|=m} \frac{1}{(2 \pi i)^m} \int_{L, \Delta_m \ni T} \tau L_P \text{d}^m \log \vartriangle \!\!z[P](T) \\
&=\sum_{m \geq 0} \sum_{|P|=m} \frac{1}{(2 \pi i)^m} \int_{L, \Delta_m \ni T} D_{[P]} \text{d}^m \log \vartriangle \!\!z[P](T) \\
&=Z(L)
\end{align}
which can be written:
\beq
Z(L)=\sum_{m \geq 0} \sum_{\substack{ [P] \\ |P|=m}} D_{[P]} \cdot \left( \frac{1}{(2 \pi i)^m} \sum_{P' \in [P]} \int_{L,\Delta_m \ni T} \text{d}^m \log \vartriangle \!\! z[P'](T) \right)
\eeq

\section{Construction of $\nabla$ from first principles}
In Section \ref{constr}  we showed $\barAXN \rightarrow X_N$ was isomorphic to $C_N \mathbb{C} \times_{S_N \mathbb{C}} \barACNC / S_N \rightarrow X_N$. In this section, we will construct the bundle $\barAXN$ over $X_N$ from first principles as well as the KZ connection $\nabla =d-\omega$ on it. Doing so, it will be manifest why the KZ connection is the natural one on a certain principal bundle from which we will construct $\barAXN \rightarrow X_N$ using the associated bundle construction. We mainly follow ~\cite{KoNom}, ~\cite{Nak} applied to our situation.

\subsection{The principal bundle $P(X_N, G)$}
The structure group $G$ for the principal bundle that we will introduce is the group $(exp\mathcal{A}^{(1)}(X_N), \cdot)$ where elements of $\mathcal{A}^{(1)}(X_N)$ are complex linear combinations of chord diagrams based at some fixed point of $X_N$ of the form $\sum_{|P|=1} \lambda_P |P\rangle$ for some complex coefficients $\lambda_P$, and $\cdot$ denotes the usual associative product of exponentials:
\beq
e^{\sum_P \lambda_P |P\rangle }\cdot e^{\sum_Q \mu_Q  |Q\rangle }=e^{\sum_P (\lambda_P + \mu_P  )|P\rangle }
\eeq
We have:
\beq
(e^{\sum_P \lambda_P |P\rangle })^{-1}=e^{-\sum_P \lambda_P |P\rangle }
\eeq
and the unit for the product is $\exp(\sum_P 0\cdot |P\rangle)=\exp(0)=1$. $G$ is a commutative group with this multiplication. We mainly follow ~\cite{NSt} in the construction of our principal bundle. We first make $G$ into a Lie group. It suffices to make $G$ into a differentiable manifold on which the product and inverse maps are differentiable. For the differentiable manifold structure, consider the following map:
\begin{align}
G &\stackrel{e}{\leftarrow} \mathbb{C}^{N(N-1)/2} \nonumber \\
e^{\sum_P \lambda_P |P\rangle} & \dashleftarrow \sum_P \lambda_P |P\rangle
\end{align}
where $\{|P\rangle\;|\;|P|=1\}$ forms a basis of $\mathbb{C}^{N(N-1)/2}$.

We now put a topology on $G$. Let $m=N(N-1)/2$. On $\mathbb{C}^m$ we put the product topology, which we denote by $\tau_m$. We can write any open set $U \in \tau_m$ as the union of open sets of the form $U_1 \times \cdots \times U_m$. For $m$ complex numbers $\lambda_i \in U_i ^{open} \subset \mathbb{C}$, $1 \leq i \leq m$, we define the following open subset of $G$:
\beq
eU_1 \times \cdots \times U_m := \{e^{\sum_i \lambda_i |P_i\rangle}\; | \; \lambda_i \in U_i,\; 1 \leq i \leq m\}
\eeq
Write $\Lambda=\sum_P \lambda_P |P\rangle$.
For simplicity we write $eU=\{e^{\Lambda}\;|\;\Lambda \in U \}$, $U=U_1 \times \cdots \times U_m$ and we have $\Lambda \in U$ if and only if $\lambda_i \in U_i$ for $ 1 \leq i \leq m$.
We denote by $\tau_G$ the collection of all such subsets of $G$.\\
\begin{topG}
$\tau_G$ defines a topology on $G$.
\end{topG}
\begin{proof}
We first show any union of elements of $\tau_G$ is still in $\tau_G$: let $I$ be any indexing set. Then we can write:
\begin{align}
\bigcup_{\substack{eU_i \in \tau_G \\ i \in I}}eU_i&=\bigcup_{\substack{ U_i \in \tau_m \\ i \in I}}eU_i \\
&=\{e^{\Lambda}\;|\; \Lambda \in \cup_{i \in I}U_i=\Omega \in \tau_m\}  \\
&=e\Omega \in \tau_G
\end{align}
We also show that $\tau_G$ is closed under finite intersections. For $p$ finite:
\begin{align}
\bigcap_{\substack{eU_i \in \tau_G \\ 1 \leq i \leq p}}eU_i &= \bigcap_{\substack{ U_i \in \tau_m \\ 1 \leq i \leq p}}eU_i \\
&=\{e^{\Lambda}\;|\;\Lambda \in \cap_{1 \leq i \leq p}U_i = \Omega \in \tau_m \} \\
&=e\Omega \in \tau_G
\end{align}

Now, since:
\beq
\mathbb{C}^m=\cup_{U \in \tau_m}U
\eeq
then:
\beq
\bigcup_{eU \in \tau_G} eU=\bigcup_{U \in \tau_m}eU=\{e^{\Lambda}\;|\; \Lambda \in \mathbb{C}^m\}=G
\eeq
thus $G \in \tau_G$ as the union of elements of $\tau_G$. It is immediate that we also have $e\emptyset =\{ e^{\Lambda}\;|\; \Lambda \in \emptyset \}=\emptyset \in \tau_G$.
\end{proof}
The map $e$ is clearly invertible, and $\mathbb{C}^m$ being a differentiable manifold, this gives $G$ the structure of a differentiable manifold. Further the inverse map in the group $G$ is clearly differentiable. Finally if we give $G\times G$ the product topology, then the product map in $G$ is also a differentiable map. This shows that $G$ is a Lie group with the above structure.\\

\subsection{The tangent space $TG$ of $G$}
We consider the principal bundle $P=P(X_N,G)$ with base space $X_N$, with fiber the Lie group $G$, and with $G$ acting on itself on the left. A connection on $P$ is valued in the Lie algebra $\mathfrak{g}$ of the Lie group $G$, which is defined to be the set of left-invariant vector fields on $P$. So we naturally have to first define what the tangent space of $G$ is. If $c:I \rightarrow G$ is a path in $G$, $f:G \rightarrow \mathbb{R}$ a function on $G$, then $X \in T_{c(0)}G$ can be defined by:
\beq
Xf|_{c(0)}:=\frac{d f(c(t))}{dt}\Big|_{t=0} \label{X}
\eeq
If we denote by  $\partial_P \in T \mathbb{C}^m$ the partial derivative in the $P$-th direction in $\mathbb{C}^m$, and by $(e^{-1})^P$ the $P$-th component of the map $e^{-1}$, then we can rewrite the right hand side of ~\eqref{X} as:
\beq
\sum_P \frac{d}{dt}[(e^{-1})^P (c(t)) ]\Big|_{t=0}\cdot \partial_P\big|_{(e^{-1})^P (c(t))}f \circ e \circ e^{-1}(c(t))\Big|_{t=0}
\eeq
which leads us to write $X$ as:
\beq
X=\sum_P X^P \partial_P
\eeq
with $X^P:=\frac{d}{dt}[(e^{-1})^P (c(t)) ]|_{t=0}$
\begin{TP}
If $c$ is the following curve in $G$:
\begin{align}
c: I & \rightarrow G \nonumber \\
t &\mapsto e^{\sum_P \lambda_P(t)|P\rangle} \nonumber
\end{align}
then the tangent vector to $c$ at time $t=0$ is the vector $X$ expressed as:
\beq
X=\sum_P \lambda_P '(0)\partial_P
\eeq
\end{TP}
\begin{proof}
It suffices to write $(e^{-1})^P(c(t))=\lambda_P(t)$, from which we immediately have:
\beq
\frac{d}{dt}[(e^{-1})^P (c(t)) ]\Big|_{t=0}=\frac{d}{dt}(\lambda_P(t))\big|_{t=0}=\lambda_P '(0)
\eeq
it follows that:
\beq
X=\sum_P  \frac{d}{dt}[(e^{-1})^P (c(t)) ]\Big|_{t=0}\cdot \partial_P = \sum_P \lambda_P '(0) \partial_P
\eeq
\end{proof}

\subsection{The Lie algebra $\mathfrak{g}$ of $G$}
The Lie algebra of $G$ is the set of its left-invariant vector fields, so we determine which of those elements of $TG$ are left-invariant. For $a,g \in G$, we have the left translation map $L_a:G \rightarrow G$ defined by $L_a(g)=ag$, with induced map $L_{a*}:T_g G \rightarrow T_{ag} G$. Since $g \in G$, we can write $g=\exp(\sum_P \lambda_P (0) |P\rangle)$ with $c(t)=\exp(\sum_P \lambda_P (t) |P\rangle)$ a curve in $G$ through $g$ at time $t=0$. Let $X$ be the tangent vector to $G$ at $g$ along this curve, $X \in T_g G$, with:
\beq
X=\sum_P X^P \partial_P|_g
\eeq
and:
\beq
X^P=\frac{d}{dt}[(e^{-1})^P (c(t))]\Big|_{t=0}=\lambda_P '(0)
\eeq
from the previous lemma. Now let $Y$ be the element of $T_{ag} G$ defined by $Y=L_{a*} X$. We have:
\beq
Y=\sum_P Y^P \partial_P|_{ag}
\eeq
with:
\beq
Y^P=\sum_Q X^Q \partial_Q |_{g} [(e^{-1})^P \circ L_a (c(t))]\Big|_{t=0}
\eeq
Let $a=\exp(\sum_P \mu_P |P\rangle)$. Then we have:
\begin{align}
Y^P&=\sum_Q \lambda_Q '(0) \frac{\partial}{\partial (\lambda_Q (t))} [(e^{-1})^P (e^{\sum_R (\mu_R + \lambda_R (t))|R\rangle})]\Big|_{t=0} \\
&=\sum_Q \lambda_Q '(0) \frac{\partial}{\partial (\lambda_Q (t))} ( \mu_P + \lambda_P (t) )\Big|_{t=0} \\
&=\sum_Q \lambda_Q '(0) \delta_{PQ} \\
&=\lambda_P '(0)
\end{align}
From which it follows that:
\beq
Y=\sum_P \lambda_P '(0) \partial_P |_{ag}
\eeq
Finally, it suffices to observe that:
\beq
X|_{ag}=\sum_P X^P (ag) \partial_P |_{ag}
\eeq
with:
\begin{align}
X^P (ag)&=\frac{d}{dt}[ (e^{-1})^P ( e^{\sum_Q (\mu_Q +\lambda_Q (t)) |Q\rangle})] \Big|_{t=0} \\
&=\lambda_P '(0)
\end{align}
It follows:
\beq
Y = X|_{ag}
\eeq
that is $L_{a*}(X|_g)=X|_{ag}$, and this for any curve through $g$, and this for all $g$ in $G$, thus the set of left-invariant vector fields of $G$ is all of $TG$, whose basis is given by $\{ \partial_P \;|\; |P|=1\}$, a basis of vector fields for $T \mathbb{C}^m$ whose fiber is isomorphic to $\mathbb{C}^m$ with basis the vectors $\{|P\rangle \;|\; |P|=1\}$, and by isomorphism this gives the basis for left-invariant vector fields. Therefore we regard the Lie algebra $\mathfrak{g}$ of $G$ as being generated by those elements. Thus any element of $\mathfrak{g}$ can be written as $\sum_P \lambda_P |P \rangle$ for some complex numbers $\lambda_P$.\\

\subsection{The cotangent space $T^*G$ of $G$}
In the same manner that we used the compact notation $\partial_P$ to denote the partial derivative in the $P$ direction, we introduce a similar notation for forms. If the $P$-th coordinate in $\mathbb{C}^m$ is denoted by $\psi_P$, then we have $\partial_P = \partial / \partial (\psi_P)$ and a basis for $T^*G$ is given by $\{ d \psi_P \;|\; |P|=1 \}$. For a function $f: G \rightarrow \mathbb{R}$ we write $df=\sum_P \partial_P f \cdot d\psi_P$.\\

\subsection{The vertical space $VP$ of $P$}
We are interested in putting a connection on the principal bundle $P$. For this purpose we first define the vertical tangent space $VP$ of $P$. Let $u$ be an element of $P=P(X_N, G)$. If $\pi$ is the projection from $P$ to the base space $X_N$ which in any local trivialization is just projection on the first component, $Z=\pi(u) \in X_N$, then $G_Z=G$ is the fiber of $P$ over $Z$. Let $A=\sum_P \lambda_P (Z) |P\rangle \in \mathfrak{g}$. Then $u.e^{tA}$ defines a curve through $u$ in $G_Z$, and if in a local trivialization we write $u=(Z,g_u)$, $g_u \in G$, then $u \cdot e^{tA}=(Z,g_{u\cdot e^{tA}})=(Z,g_u \cdot e^{tA})$. We define a tangent vector $A^{\star} \in V_u P$ as follows. For a given function $f: P \rightarrow \mathbb{R}$:
\beq
A^{\star} f(u):=\frac{d}{dt}(f(u.e^{tA}))\Big|_{t=0}
\eeq
Thus defined, $A^{\star}$ is tangent to $G_Z$ at $u$ and is therefore in the vertical subspace $V_u P$ of $T_u P$. Write $g_u=exp(\sum_P \mu_P (Z) |P\rangle (Z))$. Writing $A^{\star} = \sum_P A^{\star P} \partial_P$, it follows:
\begin{align}
A^{\star P}(P)&= \frac{d}{dt}[(e^{-1})^P \left(Z,e^{\sum_Q ( \mu_Q (Z) + t\lambda_Q (Z))|Q\rangle (Z)}\right)]\Big|_{t=0} \\
&=\frac{d}{dt}(\mu_P (Z) + t \lambda_P (Z))\Big|_{t=0} \\
&=\lambda_P (Z)
\end{align}
thereby defining the following vector space isomorphism:
\begin{align}
\mathfrak{g} & \stackrel{\star}{\longrightarrow } VP \nonumber \\
A=\sum_P \lambda_P |P\rangle & \mapsto A^{\star} = \sum_P \lambda_P \partial_P
\end{align}

\subsection{The connection $\tilde{\omega}$ on $P$}
In the same manner that we had left multiplication on $G$ we also have a right multiplication map denoted by $R$. If $a,g \in G$ then $R_a (g)=ga$, and the map $R_a$ induces a map $R_a ^*: T^* _{ua}P \rightarrow T^* _u P$. Further, we have the adjoint action of $G$ on itself defined by $ad_g h:=ghg^{-1}$ with a corresponding induced map $ad_{g*}=Ad_g:T_h G \rightarrow T_{ghg^{-1}}G$, and for $h=1$ the identity of $G$ we have a resulting adjoint map $Ad_g : \mathfrak{g} \rightarrow \mathfrak{g}$ defined by $A \mapsto Ad_g (A)=gAg^{-1}$.\\
We now define the connection one form $\tilde{\omega}$ on $P$. A connection one form $\tilde{\omega} \in \mathfrak{g}\otimes T^*P$ is a projection of $TP$ onto $VP \simeq \mathfrak{g}$, satisfying:
\begin{align}
(1)\qquad & \tilde{\omega}(A^{\star})=A \\
(2)\qquad & R_g ^* \tilde{\omega}= Ad_{g^{-1}} \tilde{\omega}
\end{align}
For our connection, we choose:
\beq
\tilde{\omega}=\sum_P |P\rangle d\psi_P
\eeq
\begin{tildeomega}
$\tilom$ is a well-defined connection on $P(X_N,G)$ over $X_N$.
\end{tildeomega}
\begin{proof}
For $A=\sum_P \lambda_P |P\rangle \in \mathfrak{g}$, we have the corresponding element of $VP$ $A^{\star}=\sum_P \lambda_P \partial_P$. Now:
\begin{align}
\tilom ( A^{\star})&=(\sum_P |P\rangle d\psi_P )(\sum_Q \lambda_Q \partial_Q) \\
&=\sum_P |P\rangle d\psi_P (\sum_Q \lambda_Q \partial_Q) \\
&=\sum_{P,Q} |P\rangle \lambda_Q d\psi_P (\partial_Q) \\
&=\sum_{P,Q} |P\rangle \lambda_Q \delta_{P,Q} \\
&=\sum_P \lambda_P |P\rangle \\
&=A
\end{align}
so that $\tilom ( A^{\star})=A$. In a local trivialization $P \simeq U \times G$, $\pi$ is just the projection on the first factor. If $\pi_G$ is the projection on the second factor, then we locally have ~\cite{Sp}: $TP \simeq T(U \times G) \simeq \pi^* TU \oplus \pi_G ^* TG$. Then $X \in T_u P$, $u=(Z,g_u)$ in a local trivialization, can be written as a sum $X=X_1 + X_2$, $X_1 \in \pi^* T_Z X_N$, $X_2 \in T_{g_u} G$, $\tilom _u (X_1)=0$, $X_2=\sum_P X^P \partial_P$. Then $\tilom_u(X)=\tilom_u (X_2)=\sum_P X^P |P\rangle$. For $g \in G$, by commutativity in $\mathcal{A}^{(1)}(X_N)$ it follows that $Ad_{g^{-1}}\tilom_u(X)=g^{-1}\tilom_u (X)g=\tilom_u (X)=\sum_P X^P |P\rangle$. On the other hand $R_g ^* \tilom_{ug} (X)=\tilom_{ug}(R_{g*}X)$, with $R_{g*}X=R_{g*}X_1 + R_{g*}X_2=R_{g*} X_1 + \sum_P Y^P \partial_P |_{ug}$, with:
\beq
Y^P =\sum_Q X^Q \partial_Q |_u ((e^{-1})^P (R_g u))
\eeq
If we let $g=\exp(\sum_P \mu_P |P\rangle)$ and $u=\left(Z,\exp(\sum_P \lambda_P |P\rangle)\right)$, then we have:
\begin{align}
Y^P&=\sum_Q X^Q \partial_Q |_u [ (e^{-1})^P \left(Z, e^{\sum_R ( \lambda_R + \mu_R)|R\rangle}\right)] \\
&=\sum_Q X^Q \partial_Q |_u ( \lambda_P + \mu_P) \\
&=\sum_Q X^Q \frac{\partial}{\partial ( \lambda_Q)}(\lambda_P + \mu_P)\\
&=X^P|_u
\end{align}
From which it follows that $R_{g*}X=R_{g*}X_1 + \sum_P X^P|_u \cdot \partial_P |_{ug}$. And:
\begin{align}
R_g ^* \tilom_{ug} (X)&=\tilom_{ug}(R_{g*}X) \\
&=\tilom_{ug} (R_{g*}X_1 + \sum_P X^P |_u \cdot \partial_P |_{ug})\\
&=\tilom_{ug} (\sum_P X^P |_u \cdot \partial_P |_{ug})\\
&=(\sum_P |P\rangle d\psi_P|_{ug})(\sum_Q X^Q|_u \partial_Q |_{ug}) \\
&=\sum_{P,Q} |P\rangle X^Q |_u \delta_{P,Q} \\
&=\sum_P X^P |_u |P\rangle \\
&=\tilom_u (X)\\
&=Ad_{g^{-1}} \tilom_u (X)
\end{align}
which completes the proof.
\end{proof}

\subsection{Local form $\omega$ of the connection one form $\tilom$}
We wish to write a local expression for $\tilom$ by pulling it back to $X_N$. Let $U$ be an open set in $X_N$, $\pi: P \rightarrow X_N$, $\pi^{-1}U$ open in $P$, on which we define $TP$ and $T^*P$. A section of $P$ at $Z \in U$ is of the form $(Z,\exp(\sum_P \psi_P (Z) |P\rangle (Z)))$. Consider a local section $\sigma$ of $P$ on $U$ for which $\psi_P (Z)=\frac{1}{2 \pi i}\logp (Z)$. Then we define:
\begin{align}
\omega &:= \sigma ^* \tilom \\
&=\sigma ^* (\sum_P |P\rangle d\psi_P) \\
&=\sum_P |P\rangle \sigma ^* d\psi_P \\
&=\sum_P |P\rangle d\psi_P (\sigma_{*})
\end{align}
For the purpose of computing $\omega$ we work with small enough basic open sets in $X_N$. Recall that the topology $\tau_N$ on $X_N$ is generated by small enough basic open sets $V=\{V_1,...,V_N\}$, $V_1, ..., V_N$ $N$ disjoint open sets in the complex plane. We can write $U$ as a union of such open sets. Let $V$ be one such small enough basic open such that $Z \in V$. Then we can locally label the $N$ defining points of $Z$ as $z_1,...,z_N$, and we locally have a basis $\partial / \partial z_k$, $1 \leq k \leq N$ of $T X_N$. Fix $k$, $1 \leq k \leq N$. Then:
\beq
\sigma_{*} \frac{\partial}{\partial z_k}=\sum_P \lambda_P \partial_P
\eeq
with:
\beq
\lambda_P=\frac{1}{2 \pi i}\frac{\partial}{\partial z_k} \logp (Z)
\eeq
so that:
\beq
\sigma_{*} \frac{\partial}{\partial z_k}=\sum_P \frac{1}{2 \pi i}\frac{\partial}{\partial z_k} \logp (Z) \partial_P
\eeq
It follows:
\begin{align}
(\sigma^* d\psi_P )(\frac{\partial}{\partial z_k})&= d\psi_P \Big( \sum_Q \frac{1}{2 \pi i}\frac{\partial}{\partial z_k} log \vartriangle \!\!z[Q] (Z) \partial_Q\Big) \\
&=\sum_Q \frac{1}{2 \pi i}\frac{\partial}{\partial z_k} log \vartriangle \!\! z[Q] (Z) d\psi_P (\partial_Q) \\
&=\sum_Q \frac{1}{2 \pi i}\frac{\partial}{\partial z_k} log \vartriangle \!\! z[Q] (Z) \delta_{P,Q} \\
&=\frac{1}{2 \pi i}\frac{\partial}{\partial z_k} \logp (Z)
\end{align}
from which we get:
\begin{align}
 \sigma^* d\psi_P &= \sum_k \frac{1}{2 \pi i}\frac{\partial}{\partial z_k} \logp (Z) dz_k \\
&=\frac{1}{2 \pi i}\dlogp
\end{align}
Finally:
\beq
\omega=\sigma^* \tilom = \frac{1}{2 \pi i}\sum_P |P\rangle \dlogp
\eeq
which is commonly referred to as the Knizhnik-Zamolodchikov connection one form. Notice that this derivation is independent of the choice of labels for the defining points of $Z$ in a basic open set.\\

\subsection{Local connection $\tiloml$ on $P$}
We now seek to construct a local form of a connection on $P$ derived from the local form $\omega$ on $X_N$, which is itself derived from the connection one form $\tilom$ on $P$, for the purpose of obtaining the Knizhnik-Zamolodchikov equation and show that upon integration we get a holonomy for that particular connection which generates the Kontsevich integral.\\

For the section $\sigma$ introduced above, $Z \in X_N$ fixed, any element $u$ of $P$ in the fiber $G_Z$ above $Z$ can be written $u=\sigma(Z)g$ for some $g \in G$. Using $g$ thus defined, consider the following element of $T_u ^* P$:
\beq
\tiloml :=-g^{-1} \pi^* \omega g + g^{-1} d_P g
\eeq
where $d_P$ is the exterior derivative on $P$.
\begin{pullbacktiloml}
we have:
\beq
\sigma^* \tiloml = -\omega
\eeq
\end{pullbacktiloml}
\begin{proof}
For $X \in T_Z X_N$, we have:
\beq
\sigma^* \tiloml (X)= \tiloml ( \sigma_* X)
\eeq
and $\sigma_* X \in T_{\sigma(Z)} P$ where $g=1$, so that:
\begin{align}
\sigma^* \tiloml (X)&=-\pi^* \omega (\sigma_* X) + d_P g (\sigma_* X) \\
&= -\omega ( \pi _* \sigma_* X) + d_P g ( \sigma_* X)
\end{align}
Now we use the fact that $\pi_* \sigma_* =id_{T_Z X_N}$ and $d_P g (\sigma_* X)=0$ since $g=1$ along $\sigma _* X$. It follows that we have:
\beq
\sigma^* \tiloml (X)=-\omega (X) +0 =-\omega(X)
\eeq
and this for any element $X$ of $T_Z X_N$ and for all $Z \in X_N$, whence the result.
\end{proof}
We now show $\tiloml$ verifies the axioms of a connection on $P$:
\begin{tilomlconn}
$\tiloml$ is a well-defined connection on $P$.
\end{tilomlconn}
\begin{proof}
First for some $A$ in $\mathfrak{g}$ of the form $A=\sum_P \lambda_P |P\rangle$, $u \in P$, and a corresponding $A^{\star}=\sum_P \lambda_P \partial_P \in V_u P$, then we show $\tiloml (A^{\star})=A$.
\begin{align}
\tiloml ( A^{\star})&= -g^{-1} \omega ( \pi_* A^{\star} )g + g^{-1} d_P g (A^{\star}) \\
&=-g^{-1} \omega (0) g + g^{-1} \frac{d}{dt} ( g_{u.e^{tA}})\Big|_{t=0} \\
&= g^{-1} \frac{d}{dt}( g_u.e^{tA})\Big|_{t=0} \\
&=g_u ^{-1}  g_u \frac{d}{dt} ( e^{tA}) \Big|_{t=0} \\
&= A
\end{align}
Further we have to show that we have $R_h ^* \tiloml = Ad_{h^{-1}} \tiloml$: let $X \in T_u P$, $h \in G$. Then we have:
\beq
R_h ^* \tiloml (X)= \tiloml (R_{h*} X)= -g_{uh} ^{-1} \omega ( \pi_* R_{h*} X) g_{uh} + g_{uh} ^{-1} d_P g_{uh} (R_{h*} X) \label{ahh}
\eeq
We now use $g_{uh}=g_u \cdot h$, $ \pi_* R_{h*} X= \pi_* X$ since $\pi R_h =\pi$, and the following:
\begin{align}
R_{h*}Xg_{uh}&=R_{h*}Xg|_{uh}\\
&=\frac{d}{dt}g_{u.c(t)}\cdot h \Big|_{t=0}\\
&=Xg_u \cdot h=d_P g_u(X)\cdot h
\end{align}
to rewrite ~\eqref{ahh} as:
\begin{align}
R_h ^* \tiloml(X)&= -h^{-1} g_u ^{-1} \omega (\pi_* X) g_u h + h^{-1} g_u ^{-1} d_P g_u (X)h \\
&=h^{-1} \tiloml (X)h \\
&=Ad_{h^{-1}}\tiloml (X)
\end{align}
and this for any $X \in T_uP$ and for any $u \in P$. This completes the proof
\end{proof}

\subsection{Horizontal lift $\tilde{\gamma}$ in $P$ of a curve $\gamma$ in $X_N$}
We construct a horizontal lift of some curve $\gamma$ in $X_N$. We do this locally. For this purpose we work in an open set $U$ that contains the curve $\gamma$. Suppose we have two open sets $U_i$ and $U_j$ in $X_N$ whose overlap contains $U$, and $\sigma_i$ and $\sigma_j$ are two sections of $P$ over $U_i$ and $U_j$ respectively. If $\tilde{\gamma}$ is a horizontal lift of $\gamma$, then for all $t \in I$ we can write $\tilde{\gamma}(t)=\sigma_j (\gamma t) g_j (\gamma t)$ and we write $g_j (\gamma t)=g_j (t)$ for short. We can pick a lift for which $g_j (0)=1$ so that $\tilde{\gamma}(0)=\sigma_j (\gamma 0)$. If $X \in T_{\gamma 0}X_N$, let $\tilde{X}$ be the horizontal lift of $X$ along $\tilde{\gamma}$. We have $\tilde{X}=\tilde{\gamma}_* X$. We would like an expression for $\tilde{X}$ that will enable us to extract the KZ equation. For this purpose we need the following lemma where we use the fact that $\tilde{\gamma}(t)=\sigma_j (t) g_j (t)=\sigma_i (t) g_i (t)$ where we have omitted the mention of $\gamma$ for the simplicity of notation.
\begin{poundeq}[\cite{Nak}]
$\tilde{X}=\sigma_{i*}X\cdot g_i (t) + (g_{i} ^{-1} dg_{i} (X))^{\star}$
\end{poundeq}
\begin{proof}
It suffices to write:
\begin{align}
\tilde{X}&= \tilde{\gamma}_* X \\
&=\frac{d}{dt}[ \tilde{\gamma} (t)]\Big|_{t=0} \\
&=\frac{d}{dt}[\sigma_i (t) g_i (t)]\Big|_{t=0}\\
&=\frac{d}{dt}\sigma_i (t)\Big|_{t=0}\cdot g_i (0) + \sigma_i (0) \cdot \frac{d}{dt}g_i (t)\Big|_{t=0} \\
&= \frac{d}{dt}[\sigma_i (t)\cdot g_i (0)]\Big|_{t=0} + \sigma_j (0)g_j (0) g_i ^{-1} (0) \frac{d}{dt}g_i (t)\Big|_{t=0} \\
&=\sigma _{i*}X\cdot g_i (0) + \frac{d}{dt}[\sigma_j (0) g_i ^{-1} (0)g_i (t)]\Big|_{t=0}
\end{align}
where we used the fact that $g_j(0)=1$. Working on the second term alone, $\sigma_j (0) g_i ^{-1} (0)g_i (t)$ is a curve in $G_{\sigma_j(0)}$, thus the whole second term is in $V_{\sigma_j(0)}P$ and we can write:
\begin{align}
\frac{d}{dt}[\sigma_j (0) g_i ^{-1}(0) g_i (t)]\Big|_{t=0}&=\Big(\frac{d}{dt}[\sigma_j (0) g_i ^{-1} (0)g_i (t)]\Big|_{t=0}\Big)^{\star}\\
&=(\frac{d}{dt}g_i (t))^{\star}\Big|_{\sigma_j (0) g_i ^{-1}(0)} \\
&=(dg_i (X))^{\star}\Big|_{\sigma_j (0) g_i ^{-1}(0)} \\
&=(g_i ^{-1}\cdot dg_i (X))^{\star}\Big|_{\sigma_j (0)}
\end{align}
which completes the proof after time reparametrization $0 \mapsto t$.
\end{proof}
As a corollary, since this result is independent of the choice of local section, we have:
\begin{poundeqsimple}
$\tilde{X}= \sigma_{*} X\cdot g (t) + ( g ^{-1} (t)\cdot dg (X))^{\star}$
\end{poundeqsimple}

\subsection{The KZ equation}
From the previous lemma we have a local expression for a horizontal vector $\tilde{X}$. Moreover, $\tilde{X} \in HP \Rightarrow \tiloml (\tilde{X})=0$.
\begin{KZeq}
The KZ equation:
\beq
dg=\omega\cdot g
\eeq
follows from the fact that $\tiloml (\tilde{X})=0$.
\end{KZeq}
\begin{proof}
We have from the previous lemma:
\beq
\tilde{X}= \sigma_{*} X\cdot g (t) + ( g ^{-1} (t) dg (X))^{\star}
\eeq
and upon applying $\tiloml$ on both sides, we get:
\begin{align}
0&=\tiloml ( \sigma_* X \cdot g(t)) + \tiloml (( g ^{-1} (t) dg (X))^{\star}) \\
\nonumber \\
&=\tiloml ( R_{g(t)*} \sigma_* X ) +  g ^{-1} (t) dg (X) \\
\nonumber \\
&= R_{g(t)}^*\tiloml (\sigma_* X) + g ^{-1} (t) dg (X)\\
\nonumber \\
&=Ad_{g^{-1} (t)} \tiloml (\sigma_* X) + g ^{-1} (t) dg (X)\\
\nonumber \\
&=g^{-1} (t) \tiloml ( \sigma_* X)g(t) + g^{-1} (t) dg (X)
\end{align}
from which we get:
\beq
-g^{-1} (t) \sigma^* \tiloml (X) g(t)=g ^{-1} (t) dg (X)
\eeq
and using the fact that $\sigma^* \tiloml=-\omega$, we get:
\beq
-g^{-1}(t) (-\omega (X))g(t)=g^{-1}(t)dg (X)
\eeq
and by multiplying on each side by $g(t)$ we get:
\beq
dg(X)=\omega(X)g
\eeq
and this being true for all $X \in T X_N$, it follows that we have the KZ equation:
\beq
dg=\omega g
\eeq
\end{proof}

\subsection{Connection on $\barAXN$ induced from $\nabla$ on $P$}
We define an action of $G$ on $P \times \mathbb{C}[[\{|Q\rangle \;|\;|Q|=1\}]]$ as follows: for $u \in P$, $\xi \in \mathbb{C}[[\{|Q\rangle \;|\;|Q|=1\}]]$ and $g \in G$, then:
\beq
g\cdot(u,\xi)=(ug, g^{-1} \xi)
\eeq
We regard the bundle $\barAXN$ as $E:=P \times \mathbb{C}[[\{|Q\rangle \;|\;|Q|=1\}]]/G$ over $X_N$ with structure group $G$. The projection $\pi_E: E \rightarrow X_N$ is defined by the projection $\pi$ on the first factor, and local trivializations are given by $\Phi: U \times \mathbb{C}[[\{|Q\rangle \;|\;|Q|=1\}]] \rightarrow \pi_E ^{-1}(U)$ for $U$ open in $X_N$. For $U$ open in $X_N$, $\sigma$ a section of $P$ over $U$, $\gamma$ a curve in $U$, $\tilde{\gamma}$ a horizontal lift of $\gamma$, $\gamma(0)=Z \in X_N$, $\tilde{\gamma}(0)=u_0$, $X \in T_Z X_N$, $s$ a section of $E=\barAXN$ over $U$, then we can write:
\beq
s(\gamma t)=[(\tilde{\gamma}(t),\xi (\gamma t))]
\eeq
for some $\eta: X_N \rightarrow \mathbb{C}[[\{|P\rangle / |P|=1\}]]$. The covariant derivative of $s(t)$ along $\gamma(t)$ at $\gamma (0)=Z$ is defined by:
\beq
\nabla_X s =\left[ (\tilde{\gamma}(0),\frac{d}{dt}\xi (\gamma (t))|_{t=0}) \right]
\eeq
For $\sigma$ fixed, we can write $\tilde{\gamma}(t)=\sigma (\gamma t) g (\gamma t)$. For $|P|=m$, $P=(P_1,\cdots , P_m)$, $|P\rangle = |P_1\rangle \cdot ... \cdot |P_m\rangle $ is the $P$-th basis vector of $\mathbb{C}[[\{|P\rangle \;|\;|P|=1\}]]$. Consider the following section of $E$:
\beq
e_P (Z)=[ (\sigma(Z),|P\rangle )]
\eeq
so that:
\beq
e_P (\gamma t)=[(\tilde{\gamma}(t)\cdot g^{-1}(t),|P\rangle )]=[(\tilde{\gamma}(t),g^{-1}(t)\cdot |P\rangle)]
\eeq
Then:
\begin{align}
\nabla_X e_P&=[(\tilde{\gamma}(0),\frac{d}{dt}(g^{-1}(t)\cdot |P\rangle )|_{t=0})] \\
& \nonumber \\
&=[(\tilde{\gamma}(0),-g^{-1}(t)\frac{d}{dt}g(t)g^{-1}(t)\cdot |P\rangle |_{t=0})] \\
& \nonumber \\
&=[(\tilde{\gamma}(0),-g^{-1}(0)dg(X)g^{-1}(0)\cdot |P\rangle )] \\
& \nonumber \\
&=[(\tilde{\gamma}(0),-g^{-1}(0)\omega (X)g(0)g^{-1}(0)\cdot |P\rangle )] \\
& \nonumber \\
&=[(\tilde{\gamma}(0),-g^{-1}(0)\omega(X)\cdot |P\rangle )] \\
& \nonumber \\
&= [(\tilde{\gamma}(0)g^{-1}(0),-\omega(X)\cdot |P\rangle )] \\
& \nonumber \\
&=[(\sigma (Z), -\omega (X) \cdot |P\rangle)]
\end{align}
that is $\nabla e_P = [(\sigma (Z), -\omega  |P\rangle )]$.
More generally, for a section $s(Z)=[(\sigma (Z), \xi (Z))]$ of $E$ with $\xi =\sum_P \lambda_P |P\rangle$, then using essentially the same computational techniques, and letting $\xi(t)=g(t)\eta ( \gamma t)$, we would find:
\begin{align}
\nabla_X s&= [(\tilde{\gamma}(0), \frac{d}{dt}(g^{-1}(t) \xi (t))|_{t=0})] \\
&=[(\tilde{\gamma}(0), -g^{-1}(0)\omega (X) \xi (0) + g^{-1}(0)\frac{d}{dt} \xi (t)|_{t=0})] \\
&=[( \tilde{\gamma}(0)g^{-1}(0), \frac{d}{dt}\xi(t)|_{t=0}-\omega (X) \xi (Z))] \\
&=[( \sigma (Z), d\xi(X) -\omega(X) \xi (Z))]
\end{align}
and in a more compact notation $\nabla s =[(\sigma (Z), (d -\omega )\xi (Z))]$. When we considered elements $\psi(Z)=\sum_P \lambda_P (Z) |P\rangle $ of $\barA(Z)$ we really meant $\sum_P \lambda_P (Z) e_P (Z)$ and:
\begin{align}
\nabla s&= [(\sigma , (d-\omega) \sum_P \lambda_P |P\rangle )] \\
&=\sum_P (d-\omega)\lambda_P e_P \\
&=(d-\omega)\sum_P \lambda_P e_P \\
&:=\nabla_{\barAXN} \psi
\end{align}
which is the KZ connection on $\barAXN$ as initially introduced.


\begin{thebibliography}{200}
\bibitem[JB]{JB}Joan Birman, \textit{Braids, Links, and Mapping Class Groups}, Annals of Mathematics Studies $\mathbf{82}$, Princeton University Press, Princeton, NJ, 1974
\bibitem[DBN2]{DBN2}Dror Bar Natan, \textit{The Fundamental Theorem of Vassiliev Invariants}, Geometry and Physics (J.E.Andersen, J.Dupont, H.Pedersen and A.Swann, eds.), Lecture Notes in Pure and Applied Mathematics $\mathbf{184}$, Marcel Dekker, New York, 1997.
\bibitem[DBN3]{DBN3}Dror Bar Natan, \textit{Vassiliev and Quantum Invariants of Braids}, arXiv:q-alg/9607001v1.
\bibitem[Sp]{Sp}M.Spivak, \textit{A Comprehensive Introduction to Differential Geometry}, Vol.I, Publish or Perish, 1999.
\bibitem[V]{V}V.A.Vassiliev, \textit{Cohomology of Knot Spaces}, Theory of Singularities and Applications (Providence)(V.I.Arnold, ed.), Amer.Math.Soc., Providence, 1990.
\bibitem[KoNom]{KoNom} S.Kobayashi, K.Nomizu, \textit{Foundations of Differential Geometry}, Interscience, New York, 1963.
\bibitem[Nak]{Nak} Mikio Nakahara, \textit{Geometry, Topology and Physics}, Institute of Physics Publishing, Bristol, 1990.
\bibitem[NSt]{NSt} N.Steenrod, \textit{The Topology of Fiber Bundles}, Princeton Landmarks in Mathematics, Princeton, NJ, 1951.
\bibitem[ChDu]{ChDu} S.Chmutov, S.Duzhin, \textit{The Kontsevich Integral}, Acta Applicandae Mathematicae \textbf{66}:155-190, 2001.
\bibitem[Koh]{Koh}Toshitake Kohno, \textit{Conformal Field Theory and Topology}, AMS Translations of Mathematical Monographs, Volume 210, 1998.
\bibitem[ChI]{ChI}Kuo-Tsai Chen, \textit{Connections, Holonomy and Path Space Homology}, Proceedings of Symposia in Pure Mathematics Vol.27 (1975).
\bibitem[ChII]{ChII}Kuo-Tsai Chen, \textit{Iterated Path Integrals}, Bulletin of the American Mathematical Society Vol.83 \textbf{5}, September 1977.
\bibitem[BN]{DBN} Dror Bar Natan, \textit{On The Vassiliev Knot Invariants}, Topology \textbf{34} (1995) \textit{423-472}.
\bibitem[K]{K} Maxim Kontsevich, Advances in Soviet Mathematics, Volume \textbf{16}, Part 2, 1993.
\bibitem[LM]{lemu} Le Tu Quoc Thang and Jun Murakami, \textit{The Universal Vassiliev-Kontsevich Invariant for Framed Oriented Links}, arXiv:hep-th/9401016v1, January 1994.
\bibitem[A]{A} J.W.Alexander, \textit{A Lemma on Systems of Knotted Curves}, Proc. Nat. Acad. Sci. USA, \textbf{9}, \textit{93-95}.
\bibitem[GS]{GS}R.E.Gompf, A.I.Stipsicz, \textit{4-Manifolds and Kirby Calculus}, Graduate Studies in Mathematics, Vol. \textbf{20}.
\end{thebibliography}
\end{document}